\documentclass[a4paper,11pt,draft]{amsart}
\usepackage[margin=16mm]{geometry}
\usepackage{amssymb, amsmath, amsthm, amscd}

\usepackage{amsmath}
\allowdisplaybreaks[4]

\usepackage[T1]{fontenc}
\usepackage{eucal,mathrsfs,dsfont}
\usepackage{color}
\usepackage{mathrsfs}

\definecolor{mno}{rgb}{0.5,0.1,0.5}

\newcommand{\R}{\mathbb R}

\newcommand{\Pp}{\mathbb P}
\newcommand{\Ee}{\mathbb E}

\newcommand{\I}{\mathbf 1}

\def\<{\langle}
\def\>{\rangle}

\newtheorem{theorem}{Theorem}[section]
\newtheorem{lemma}[theorem]{Lemma}
\newtheorem{proposition}[theorem]{Proposition}

\theoremstyle{definition}

\newtheorem{remark}[theorem]{Remark}

\begin{document}
\allowdisplaybreaks
\title[Schr\"{o}dinger operators with decaying potentials]
{\bfseries Sharp two-sided heat kernel estimates for Schr\"{o}dinger operators
with decaying potentials
}

\author{Xin Chen\qquad Jian Wang}

\date{}
\thanks{\emph{X.\ Chen:}
   School of Mathematical Sciences, Shanghai Jiao Tong University, 200240 Shanghai, P.R. China. \texttt{chenxin217@sjtu.edu.cn}}
    \thanks{\emph{J.\ Wang:}
    School of Mathematics and Statistics \& Key Laboratory of Analytical Mathematics and Applications (Ministry of Education) \& Fujian Provincial Key Laboratory
of Statistics and Artificial Intelligence, 
Fujian Normal University, 350117 Fuzhou, P.R. China. \texttt{jianwang@fjnu.edu.cn}}
\maketitle

\begin{abstract} We establish
global
two-sided heat kernel estimates (for full time and space) of the Schr\"odinger operator $-\frac{1}{2}\Delta+V$ on $\R^d$, where the potential $V(x)$ is locally bounded and behaves like $c|x|^{-\alpha}$ near infinity with $\alpha\in (0,2)$ and $c> 0$, or with $\alpha>0$ and $c<0$.
Our results
improve all known results in the literature, and it seems that the current paper is the first one
where consistent
two-sided heat kernel bounds for the long range
potentials are established.

\medskip

\noindent\textbf{Keywords:} heat kernel; Schr\"odinger operator; decaying potential;  the Markov property; the Duhamel formula
\medskip

\noindent \textbf{MSC 2010:} 35K08; 35J10; 60J65.
\end{abstract}
\allowdisplaybreaks

\section{Introduction} Semigroups and heat kernels for Schr\"{o}dinger operators involving the Laplacian operator $\Delta$ are now classical topics --- they were studied in many papers by both analytic and probabilistic methods, see e.g. \cite{BDS, CZ, Simon1} and the references therein.
Let $\mathcal L^V$ be the Schr\"odinger operator on $\R^d$ as follows:
\begin{equation}\label{e1-1}
\mathcal L^V=-\frac{1}{2}\Delta+V,
\end{equation}
where $\Delta:=\sum_{i=1}^d \frac{\partial^2}{\partial x_i^2}$ denotes the Laplacian operator on $\R^d$, and the potential
$V:\R^d\to \R$ is locally bounded and is also bounded from below.
Then, it is well known that (see e.g.\ \cite[Theorem A.2.7]{Simon1}) there exists a
Schr\"odinger semigroup $\{T_t^V\}_{t\ge 0}$ associated with the operator $\mathcal L^V$ defined by \eqref{e1-1}, so that the following Feynman-Kac formula holds:
\begin{equation}\label{e1-4}
T_t^V f(x)=\Ee_x\left[f(B_t)\exp\left(-\int_0^t V(B_s)\,ds\right)\right],\quad f\in C_b(\R^d),
\end{equation}
where $\{B_t\}_{t\ge 0}$ is the standard $\R^d$-valued Brownian motion, and $\Pp_x$ and $\Ee_x$ denote
the probability and the expectation of $\{B_t\}_{t\ge 0}$ with the initial value $x\in \R^d$ respectively.
Moreover, by \cite[Section B7]{Simon1},  there exists a jointly continuous density $p:\R_+\times \R^d\times \R^d \to \R_+$  (with respect to the Lebesgue measure)  associated
with $\{T_t^V\}_{t\ge 0}$ such that
\begin{equation}\label{e1-5}
T_t^V f(x)=\int_{\R^d}p(t,x,y)f(y)\,dy,\quad f\in C_b(\R^d),\ t>0,\ x\in \R^d.
\end{equation} In the literature, $p(t,x,y)$ is usually called the heat kernel of the Schr\"odinger semigroup $\{T_t^V\}_{t\ge 0}$ .

Throughout this paper, let $q(t,x,y)$ denote the transition density of the standard $d$-dimensional Brownian motion $\{B_t\}_{t\ge0}$ (which is also called the heat kernel of the Laplacian operator $\frac{1}{2}\Delta$). It is well known that
\begin{equation}\label{e1-6}
q(t,x,y)=(2\pi t)^{- {d}/{2}}\exp\left(-\frac{|x-y|^2}{2t}\right),\quad t>0,\ x,y\in \R^d.
\end{equation}
The notation $f\asymp g$ means that there are positive constants $c_1,c_2,c_3$ and $c_4$ such that $c_1f(c_2x)\le g(x)\le c_3f(c_4x)$, while $f\simeq g$ means that there are constants $c_5$ and $c_6>0$ such that $c_5f(x)\le g(x)\le c_6f(x)$.

In this paper, we will establish two-sided estimates of heat kernel $p(t,x,y)$ associated with the Schr\"odinger operator $\mathcal L^V$ with decaying potential $V$, i.e.,
the function $V:\R^d \to \R$ satisfies that
$\lim_{|x|\to \infty}V(x)=0$. For the Schr\"odinger operator $\mathcal L^V$ with decaying potential, there are
fruitful results for its spectrum and ground state, including the existence of finite or countably infinite discrete spectrum,
(dense sets of) embedded eigenvalues or singular continuous spectrum, existence or non-existence of zero-energy eigenfunction or bounded
ground state, and so on. For example, it was shown in
\cite{DS,FO,SW} that if $V(x)\simeq -|x|^{-\alpha}$ near infinity for
some $\alpha\in (0,2)$, then zero was not an eigenvalue for $\mathcal L^V$; while for the positive potential $V(x) \simeq |x|^{-\alpha}$ near infinity,
\cite{JK,Mu,N,Y} proved that
the behaviors of ground state were totally different between the case $\alpha\in (0,2)$
(which was called
long range)
 and $\alpha\in (2,+\infty)$ (which was called very short range).
The readers are also referred to \cite{BY,CK,DK,DS,EK,KT,RS,S} for more details on the background and properties of the spectrum
 for the Schr\"odinger operator $\mathcal L^V$ with decaying potentials.

In order to investigate ground state of $\mathcal L^V$, it is natural
to study estimates for associated heat kernel,
which are also interesting of their own.
Most efforts have been devoted to obtaining
sufficient conditions on decaying potentials under which the Sch\"odinger heat kernel enjoys
two-sided Gaussian-type estimates
with
the same form as that of $q(t,x,y)$, i.e., as the case that
$V \equiv 0$.
For example, according to \cite[Theorem 1.1, Remark 1.1, Remark 1.2 and Proposition 2.1]{Z3},
when $V\ge0$ such that $V\in L^{r_1}(\R^d)\cap L^{r_2}_{{\rm loc}}(\R^d)$ for some $r_1>d/2$ and $1<r_2<d/2$,  there exists a constant $c_0>0$ such that
\begin{equation}\label{e1-7}
c_0q(t,x,y)\le p(t,x,y)\le q(t,x,y), \quad x,y\in \R^d,\ t>0;
\end{equation}
when $V\le 0$ such that  $V\in L^{r_1}(\R^d)$ for some $r_1>d/2$, for every $T>0$ there exist constants $c_1(T)$ and $c_2(T)>0$ so that
\begin{equation}\label{e1-7a}
c_1(T)q(t,x,y)\le p(t,x,y)\le c_2(T)q(t,x,y),\quad x,y\in \R^d,\ 0<t\le T.
\end{equation}
\eqref{e1-7} and \eqref{e1-7a} are called \emph{sharp Gaussian-type estimates} and \emph{locally sharp Gaussian-type estimates}, respectively.
 A typical example for the assertion \eqref{e1-7} is a positive and locally bounded function $V$ that belongs to $L^{r_1}(\R^d)$ with $r_1>d/2$ and decays rapidly faster than quadratic decay (for example, behaves like $|x|^{-\alpha}$ near infinity for some $\alpha\in (2,+\infty)$). See also \cite{BDS} for the most recent progress
 concerning this problem.
We note that, sufficient conditions for \emph{plain Gaussian-type estimates},  i.e.,
 $p(t,x,y)\asymp q(t,x,y)$ for all $ x,y\in \R^d$ and $t>0$, were also given in \cite{Lis,MS1,MS2}.
 The readers can further see \cite{BDS, BJS, BS,Lis, S, Z2} for history and the details on this topic. We shall emphasize that due to the presence of negative eigenvalues, sharp full time bounds of $p(t,x,y)$ when $V<0$ are different from
 those
 when $V\ge0.$

The two-sided Gaussian estimates may not hold for decaying potential
$V(x)\simeq \pm |x|^{-\alpha}$ near infinity when $\alpha\in (0,2]$, which is totally different from
the case that $\alpha\in (2,+\infty)$. In particular, for critical decaying potential $V(x)=\pm {\lambda}{|x|^{-2}}$ with $\lambda>0$ (i.e., Hardy potential),
explicit estimates for $p(t,x,y)$ have been established in \cite{BFT,FMT,IKO,MS1,MS2}, where the associated estimates
were
no longer Gaussian-type, and some extra polynomial decay or growth terms
would
appear.  We note that the proof in \cite{BFT,FMT,IKO,MS1,MS2} essentially
uses
the scaling invariant property of $\mathcal L^V$, which was ensured by the special form of Hardy potential $V(x)=\pm {\lambda}{|x|^{-2}}$
and the scaling property of Brownian motion.
It seems difficult to extend these arguments to general decaying potential $V(x)\simeq \pm |x|^{-\alpha}$, mostly due to the fact that the scaling invariant property does not hold.
Recently two-sided estimates for Green functions corresponding to fractional Schr\"odinger operators with decaying potential
 have been established in \cite{KL}.
 Actually, though there are
developments on global two-sided estimates for Schr\"odinger heat kernels
with confining potential (i.e., $V (x)\to \infty$ as $|x|\to \infty$) (see \cite{BK,CW}), the corresponding estimates for $\mathcal L^V$ with decaying potential $V(x)\simeq \pm |x|^{-\alpha}$ are quite limited.
As far as we know, \cite{Z1,Z2} have tackled this problem, which proved that the associated heat kernel $p(t,x,y)$
was
the multiple of a standard Gaussian-type estimate with a weighted function. See also \cite{DS91} for the early study of Schr\"{o}dinger heat kernel bounds for the potential
$V(x)\simeq -|x|^{-\alpha}$ near infinity with $\alpha\in (0,2)$.
In particular, these results that fit to the present setting can be rephrased as follows.

\begin{theorem}\label{Z1} {\bf(\cite[Theorems 1.1 and 1.2]{Z1})}
Suppose that
$d\ge 3$
and that there exist constants $\alpha>0$ and $K_1,K_2>0$ such that
$$
K_1(1+|x|)^{-\alpha}\le V(x)\le K_2(1+|x|)^{-\alpha},\quad x\in \R^d.$$  Then the following
statements hold.
\begin{itemize}
\item [(1)] If $\alpha\in (0,2)$, then there exist constants $C_i$ $(i=1,\cdots, 4)$
so that for all $x,y\in \R^d$ and $t>0$,
\begin{equation}\label{e:z1}\begin{split}
C_1t^{-d/2}&\exp\left(-\frac{C_2|x-y|^2}{t}\right)\exp\left(-C_2\left(\frac{t}{(1+\max\{|x|,|y|\})^{\alpha }}\right)\right)\\
&\le p(t,x,y)\\
&\le C_3t^{-d/2}\exp\left(-\frac{C_4|x-y|^2}{t}\right)\exp\left(-C_4\left(\frac{t}{(1+\max\{|x|,|y|\})^{\alpha }}\right)^{(2-\alpha)/4}\right).\end{split}\end{equation}

\item [(2)] If $\alpha=2$, then there exist positive constants $C_i$ $(i=5,\cdots, 8)$ and $\theta_i$ $(i=1,2)$
so that for all $x,y\in \R^d$ and $t>0$,
\begin{equation}\label{Z1-1}
\begin{split}
 C_5t^{-d/2}\exp\left(-\frac{C_6|x-y|^2}{t}\right)w_1(t,x)w_1(t,y)
&\le p(t,x,y)\\
&\le C_7t^{-d/2}\exp\left(-\frac{C_8|x-y|^2}{t}\right)w_2(t,x)w_2(t,y),
\end{split}
\end{equation}
where $$w_i(t,x):=\left(\max\left\{\frac{t}{(1+|x|)^2},1\right\}\right)^{-\theta_i},\quad i=1,2.$$

\item [(3)] If $\alpha\in (2,+\infty)$, then for all $x,y\in \R^d$ and $t>0$,
\begin{equation*}
p(t,x,y)\asymp q(t,x,y).
\end{equation*}
\end{itemize}
\end{theorem}

\begin{theorem}\label{Z2} {\bf(\cite[Theorem 1.3]{Z1} and \cite[Theorem 1.1]{Z2})}
Suppose that $d\ge 3$ and that there are constants $\alpha>0$ and $K_1,K_2>0$ so that $$-K_2(1+|x|)^{-\alpha}\le V(x)\le -K_1(1+|x|)^{-\alpha},\quad x\in \R^d.$$ Then the following
statements hold.
\begin{itemize}
\item[{\rm (1)}] If $\alpha\in (0,2)$, then there exist constants $C_i>0$ $(i=1,\cdots, 3)$  so that for all $x,y\in \R^d$ and $t>0$,
\begin{equation}\label{e:z2}\begin{split}
p(t,x,y)\ge C_1t^{-d/2}\exp\left(-\frac{C_2|x-y|^2}{t}\right)\exp\left(C_3\left(\frac{t}{(1+\min\{|x|,|y|\})^{\alpha }}\right)^{(2-\alpha)/4}\right).\end{split}\end{equation}

\item [(2)] If $\alpha=2$ and $0<K_2<(d-2)^2/4$, then there exist positive constants $C_i$ $(i=4,\cdots, 7)$ and $\theta_i$ $(i=3,4)$
so that for all $x,y\in \R^d$ and $t>0$,
\begin{equation}\label{Z2-1}
\begin{split}
C_4t^{-d/2}\exp\left(-\frac{C_5|x-y|^2}{t}\right)w_3(t,x)w_3(t,y)
&\le p(t,x,y)\\
&\le C_6t^{-d/2}\exp\left(-\frac{C_7|x-y|^2}{t}\right)w_4(t,x)w_4(t,y),
\end{split}
\end{equation}
where $$w_i(t,x):=\left(\max\left\{\frac{t}{(1+|x|)^2},1\right\}\right)^{\theta_i},\quad i=3,4.$$  Furthermore, if $\alpha=2$ and $K_1>(d-2)^2/4$, then  there exist constants $C_i>0$ $(i=8,\cdots, 11)$ so that for all $x,y\in \R^d$ and $t>0$,  \begin{equation}\label{e:z3}\begin{split} p(t,x,y)\ge   C_8t^{-d/2}\exp\left(-\frac{C_9|x-y|^2}{t}\right)\exp\left(C_{10}t-\frac{C_{11}|x|^2}{t}\right). \end{split}\end{equation}

\item[{\rm(3)}] If
$\alpha>0$,
then there exists a constant $K_0>0$ such that, when $K_1>K_0$, \eqref{e:z3} holds as well possibly with different constants.
\end{itemize}
\end{theorem}
According to Theorem \ref{Z1}, we know that when $V(x)\simeq |x|^{-\alpha}$ near infinity with $\alpha\in (0,2)$, the weighted function is
bounded and enjoys exponential decay. However, it was pointed out in \cite[Remark 1.2]{Z1} that the powers involved in the exponential term in both hand sides of \eqref{e:z1} are inconsistent, and it is left an open problem to see {\it what is the optimal power on the weight function}. Meanwhile, heat kernel bounds of $p(t,x,y)$ for negative potentials $V(x)\simeq -|x|^{-\alpha}$ near infinity with $\alpha>0$ are quite different from those for positive potentials.  For example,
by Theorem \ref{Z2}, one can see that when
$\alpha\in (0,2)$
only lower bounds are known in the literature, and \emph{the corresponding upper bounds are still unknown}; see \cite[Remark 1.3]{Z1}.
On the other hand,
\cite[Remark 3.1]{Z2} showed
that
when $\alpha\in (2,+\infty)$,
if the Schr\"odinger operator $\mathcal L^V$ is not nonnegative and negative eigenvalues exist, then the associated heat kernel $p(t,x,y)$ is bounded below by a standard Gaussian times $e^{c_*t}$ with $c_*>0$. In other words, the growth rate of $p(t,x,y)$ is comparable with the heat kernel of $-\frac{1}{2}\Delta-c_*$ at large time, regardless how fast $V$ decays near infinity.

 \ \

According to all known results above, it is especially interesting to find global optimal bounds of the heat kernel for the Schr\"odinger operator $\mathcal L^V$
with decaying potentials, in particular with the potentials that are not faster than quadratic decay, i.e., $V(x)\simeq \pm |x|^{-\alpha}$ with
$\alpha\in (0,2)$. As mentioned before,
such kind of potential is usually called long range potential in the literature. To emphasize the importance of the study on long range potentials, we mention that the Coulomb potential belongs to this range.
The purpose of this paper is twofold. One is to establish two-sided estimates of $p(t,x,y)$ for positive decaying potentials $V$ so that the following condition is satisfied:
\begin{itemize}\it
\item {\bf Case 1}: there exist constants $\alpha\in (0,2)$ and $K_1,K_2>0$ such that
\begin{equation}\label{e1-1a}
K_1(1+|x|)^{-\alpha}\le V(x)\le K_2(1+|x|)^{-\alpha},\quad x\in \R^d.
\end{equation}\end{itemize}
The other is to study qualitatively sharp global heat kernel bounds for negative decaying potentials $V$ that satisfy the following condition:
\begin{itemize}\it
\item {\bf Case 2}: there exist constants $\alpha>0$ and $K_1$, $K_2>0$ such that
\begin{equation}\label{e1-1a-}
-K_2(1+|x|)^{-\alpha}\le V(x)\le -K_1(1+|x|)^{-\alpha},\quad x\in \R^d.
\end{equation}
\end{itemize}

\ \

The main results of this paper are as follows.

\begin{theorem}\label{t1-1}
Suppose that \eqref{e1-1a} holds and $d\ge 2$. Then, for any $x,y\in \R^d$ and $t>0$ it holds that
\begin{equation}\label{t1-1-1}
p(t,x,y)\asymp q(t,x,y)\left[\exp\left(-\min\left(\frac{t}{\left(1+\max\{|x|,|y|\}\right)^\alpha}, t^{(2-\alpha)/(2+\alpha)}\right)\right)\right].
\end{equation}
\end{theorem}

\ \

\begin{theorem}\label{t1-2}
Suppose that \eqref{e1-1a-} is true. Then, the following statements hold.
\begin{itemize}
\item [(1)] If \eqref{e1-1a-} holds for some $\alpha\in (0,2)$, then
for any $x,y\in \R^d$ and $t>0$,
\begin{equation}\label{t1-1-1-}
\begin{split}
&C_{1}t^{- {d}/{2}}\exp\left(-\frac{C_2|x-y|^2}{t}\right)\exp\left(\max\left\{\frac{C_3t}{(1+\min\{|x|,|y|\})^\alpha},C_3t-\frac{C_4(1+\min\{|x|,|y|\})^2}{t}\right\}\right)\\
&\le p(t,x,y)\\
&\le C_{5}t^{- {d}/{2}}\exp\left(-\frac{C_6|x-y|^2}{t}\right)\exp\left(\max\left\{\frac{C_7t}{(1+\min\{|x|,|y|\})^\alpha},C_7t-\frac{C_8(1+\min\{|x|,|y|\})^2}{t}\right\}\right).
\end{split}
\end{equation}
\item [(2)] Suppose that \eqref{e1-1a-} holds for some $\alpha\in [2,+\infty)$ and $K_1,K_2>0$. Then there exists a constant $K_*>0$ so that, if $K_1\ge K_*$, then
the two-sided estimates \eqref{t1-1-1-} are still true for every $x,y\in \R^d$ and $t>0$.

\item [(3)] Suppose that $d\ge 3$, and that \eqref{e1-1a-} holds for some $\alpha\in (2,+\infty)$ and $K_1,K_2>0.$ Then there exists $K^*>0$ such that, if $K_2\le K^*$, then
for any $x,y\in \R^d$ and $t>0$,
$$
p(t,x,y)\asymp q(t,x,y),
$$
i.e., $p(t,x,y)$ enjoys the global two-sided plain Gaussian-type estimates.
\end{itemize}
\end{theorem}

Furthermore, according to Theorem \ref{t1-1}, we know that the Green function associated with the Sch\"odinger operator $\mathcal L^V$ is well defined
for  positive potentials $V(x)\simeq (1+|x|)^{-\alpha}$ with $\alpha\in (0,2)$. More explicitly, in this case we have the following two-sided estimates for the Green function.

\begin{proposition}\label{t1-3}
Suppose that \eqref{e1-1a} hold and $d\ge 2$. Then, the Green function $G(x,y):=\int_0^\infty p(t,x,y)\,dt$ is well defined, and satisfies the following two-sided estimates
\begin{equation}\label{t1-3-1}
\begin{split}
G(x,y)&\asymp |x-y|^{-(d-2)}\exp\left(-\frac{|x-y|}{\left(1+\max\{|x|,|y|\}\right)^{\alpha/2}}\right)\cdot
\begin{cases}
1,\ \ & d\ge 3,\\
1+\left(\log\left(\frac{\left(1+\max\{|x|,|y|\}\right)^{\alpha/2}}{|x-y|}\right)\right)_+,\ &d=2.
\end{cases}
\end{split}
\end{equation}
\end{proposition}

Below we make some comments on Theorems \ref{t1-1} and \ref{t1-2}.

\begin{itemize}
\item [(i)] Since in our paper the potential $V$ is always bounded, it is easily seen from \eqref{e1-4} that there are constants $c_1,c_2>0$ so that
\begin{equation}\label{note-}e^{-c_1 t}q(t,x,y)\le p(t,x,y)\le q(t,x,y) e^{c_2t},\quad t>0,x,y\in \R^d.
\end{equation}
Hence, the novelty of Theorems \ref{t1-1} and \ref{t1-2} is to present the explicit relation between $p(t,x,y)$ and $q(t,x,y)$. In particular, Theorem \ref{t1-1} answers the open problem mentioned in \cite[Remark 1.2]{Z1},
which improves \cite[Theorems 1.1 and 1.2]{Z1} with explicit and sharp multiple weighted functions for positive
 potential $V(x)\simeq |x|^{-\alpha}$ near infinity with $\alpha\in (0,2)$. In particular,
 the estimate \eqref{t1-1-1} reveals that
 the weighted function is with the order $\exp\left(-\frac{t}{(1+\max\{|x|,|y|\})^\alpha}\right)$ when
 $t\le \left(1+\max\{|x|,|y|\}\right)^{1+\alpha/2}$; while for $t> \left(1+\max\{|x|,|y|\}\right)^{1+\alpha/2}$, the
 weighted function is with the order $\exp(-t^{(2-\alpha)/(2+\alpha)})$.
It immediately indicates the local behaviour of potential
plays a dominant role
in the Schr\"odinger heat kernel estimates when $t$ is less than the transition time $\left(1+\max\{|x|,|y|\}\right)^{1+\alpha/2}$. When $t$ is larger than the transition time $\left(1+\max\{|x|,|y|\}\right)^{1+\alpha/2}$,
the Schr\"odinger heat kernel has \emph{sub-exponential decay} with the exponent
$\frac{2-\alpha}{2+\alpha}$, which shows the global effect of the potential. In particular, this assertion improves
both the upper and the lower bounds in \cite{Z1}.
To the best of our knowledge, such explicit estimates with sub-exponential decay for large time
are not found before in existing references.

\item [(ii)] \eqref{t1-1-1-} in Theorem \ref{t1-2} provides consistent two-sided estimates
for Schr\"odinger heat kernel with the negative potential $V(x)\simeq -|x|^{-\alpha}$ near infinity when $\alpha\in (0,2)$.
 In particular, besides the upper bound, we give consistent two-sided estimates at medium time,  which has a better lower bound than that in \eqref{e:z2} and reveals the effect of local behavior of potential.
It seems that the current paper is the first one in that two-sided consistent (global) heat kernel bounds for the long range negative potential are established.
Moreover, different from
the estimates in \cite[Theorem 1.1]{Z2} where a lower bound of $K_1$ in \eqref{e1-1a-} is required to
ensure the exponential growth of heat kernel at large time, we have removed this restriction on $K_1$
(or the restriction that $\mathcal L^V$ is not nonnegative definite)
for $\alpha\in (0,2)$. This means that two-sided estimates \eqref{t1-1-1-}  and so the exponential growth for large time hold for all
$\alpha\in (0,2)$, as long as \eqref{e1-1a-} is true for some positive constants $K_1,K_2$.

\item [(iii)] For negative potential $V(x)\simeq - |x|^{-\alpha}$ near infinity with $\alpha\in (2,+\infty)$, Theorem \ref{t1-2}(2) is a qualitative version of Theorem \ref{Z2}(2) and \cite[Remark 3.1]{Z2}. Note that in Theorem \ref{t1-2}(2) the condition that $K_1\ge K_*$ is needed, and roughly speaking it ensures that
 the Schr\"odinger operator $\mathcal L^V$ is not nonnegative definite and negative eigenvalues exist.
 Furthermore, according to \cite[Theorem 1.1, Remark 1.1, Remark 1.2 and Proposition 2.1]{Z3}, the assertion in Theorem \ref{t1-2}(3) for all $x,y\in \R^d$ and $0<t\le T$ with any fixed $T>0$
 holds true with some constants depending on $T$. Hence
the novelty here is that the
 global plain Gaussian-type estimates
 hold (for all the time).
 According to Theorem \ref{t1-2} it holds that when $V(x)=-{K}{(1+|x|)^{-\alpha}}$ with $\alpha\in (2,+\infty)$, there exist
 positive constants $K_1^*\le K_2^*$ such that estimates \eqref{t1-1-1-} are true for $K\in (0,K_1^*)$, while
 plain Gaussian-type estimates hold for $(K_2^*,+\infty)$. Though we do not know whether $K_1^*=K_2^*$ is true or not, we believe that it is related to the so-called Hardy inequality.

 \item [(iv)] For the critical case $V(x)\simeq \pm |x|^{-2}$ near infinity, Theorem \ref{Z1}(2) and Theorem \ref{Z2}(2)
 have proved that the weighted function (multiplying Gaussian-type estimate) is a polynomial function of variable
 $t$, $|x|$ and $|y|$ with exponents depending on the constants $K_1$, $K_2$ in \eqref{e1-1a} and \eqref{e1-1a-}. Even
 for the potential $V(x)=\pm{K}{(1+|x|)^{-2}}$, we do not know what is sharp exponent in the weighted function. (As mentioned before, this is completely different from the case that $V(x)=\pm {K}{|x|^{-2}}$, where the scaling invariant property holds). For example, in order
 to obtain sharp exponent in the weighted function for potential $V(x)={K}{(1+|x|)^{-2}}$, it seems that
 we have to combine more accurate chain arguments (e.g. see the proofs of Lemmas \ref{l3-2} and \ref{l3-4} below) with
 the Hardy inequality.

\item [(v)] Let us briefly make some comments on new ideas for proofs of Theorems \ref{t1-1} and \ref{t1-2}, which
mostly rely on the probabilistic approaches. Compared with analytic methods,
the strong Markov property of Brownian motion will be frequently used, which is a useful tool to study
local and global interactions between Brownian motion and the potential $V$.
In particular, the proof of upper bounds in Theorem \ref{t1-1} is based on the Feynman-Kac formula \eqref{e1-4} and the upper bounds of $T_t^V1(x)$ (that is proved by applying the stopping time arguments with suitably chosen
exiting
domain), while for lower bounds we will use the probabilistic expression \eqref{l2-4-2} of the heat kernel of $p(t,x,y)$, the explicit distribution of $(\tau_{B(x,r)},B_{\tau_{B(x,r)}})$, as well as the Dirichlet heat kernel estimates \eqref{l2-1-1}. Meanwhile, for $t$ less than the transition
 time $\left(1+\max\{|x|,|y|\}\right)^{1+\alpha/2}$ we introduce two different new chain arguments to handle lower bounds (i.e., off-diagonal estimates) of $p(t,x,y)$; see Lemmas \ref{l3-2} and \ref{l3-4}. Concerning the proof of Theorem \ref{t1-2}, similar idea (with some necessary modifications for off-diagonal estimates) of the proof for Theorem \ref{t1-1} is efficient for the upper bounds when
 $\alpha\in (0,2)$.
 On the other hand,  when $\alpha>2$, to obtain Theorem \ref{t1-2}(3) we will make full use of the classical Duhamel formula. (Note that, since in this case $V\le0$, $p(t,x,y)\ge q(t,x,y)$ holds trivially). The lower bounds of Theorem  \ref{t1-2} are again proved
 by the Markov property of $p(t,x,y)$ as well as the Dirichlet heat kernel estimates \eqref{l2-1-2}.
\end{itemize}

Furthermore, we have some remarks on the effectiveness and generalizations of Theorems \ref{t1-1} and \ref{t1-2}.

\begin{remark}
\begin{itemize}
\item[{\rm(i)}] According to the proof below we can see that upper heat kernel bounds in Theorem \ref{t1-1} still hold true when $d=1$, and the condition
that $d\ge 2$ is necessary due to the chain arguments used in
proofs of Lemmas \ref{l3-2} and \ref{l3-4}.
    \item[{\rm(ii)}] By the proof or the comparison argument it is easy to see that
    upper bounds of heat kernel
    in Theorem \ref{t1-1} hold when the lower bound of the potential $V$ in \eqref{e1-1a} is satisfied, i.e., when the first inequality in \eqref{e1-1a} is satisfied. Similarly, lower bounds of heat kernel
      in Theorem \ref{t1-1} are satisfied, when the upper bound of the potential $V$ in \eqref{e1-1a} is fulfilled. Such conclusion still works for Theorem \ref{t1-2}.
    \item[{\rm(iii)}] As explained in \cite{CW}, it will be clear that Theorems \ref{t1-1} and \ref{t1-2} still hold
    when $\frac{1}{2}\Delta$ is replaced by some more general generators where the Dirichlet heat kernel estimates \eqref{l2-1-1} and \eqref{l2-1-2}
    as well as some volume doubling conditions are valid; for example,  a uniformly elliptic operator with bounded measurable coefficients in $\R^d$, or when $\R^d$ is replaced by a $d(\ge3)$-dimensional complete noncompact manifold with nonnegative Ricci curvature outside a compact set.
    \end{itemize}
\end{remark}

\ \

The rest of the paper is arranged as follows. In the next section we present the proofs of Theorem \ref{t1-1} and Proposition \ref{t1-3}, and in Section 3 we give the proof of Theorem \ref{t1-2}. Some known lemmas that are used in the proofs above are collected in the Appendix of the paper.

\section{The case for positive potentials}
This section is devoted to the proofs of Theorem \ref{t1-1} and Proposition \ref{t1-3} concerning the positive potential $V$, which is split into four subsections. Note that in this case, since $V\ge0$, it follows from \eqref{e1-4} that
$$p(t,x,y)\le q(t,x,y),\quad t>0, x,y\in \R^d.$$
\subsection{Preliminaries} Throughout this section, \emph{we always assume that \eqref{e1-1a} holds}, and define
\begin{equation}\label{e1-2a}
t_0(s):=(1+s)^{1+ {\alpha}/{2}},\quad s\ge 0.
\end{equation} For every domain $U \subset \R^d$, denote by $\tau_U:=\inf\{t>0: B_t \notin U\}$ the first exit time
from $U$ for Brownian motion $\{B_t\}_{t\ge 0}$.
For each $x\in \R^d$ and $r>0$, set $B(x,r):=\{z\in \R^d: |z-x|<r\}$.
The following lemma is crucial for the proof of upper bound estimates in Theorem  \ref{t1-1}.
\begin{lemma}\label{l2-2}
There exist constants $C_1,C_2>0$ so that for all $t>0$ and $x\in \R^d$,
\begin{equation}\label{l2-2-1}
\begin{split}
T_t^V1(x)\le C_1\left[\exp\left(-\frac{C_2t}{(1+|x|)^\alpha}\right)+\exp\left(-C_2t^{(2-\alpha)/(2+\alpha)}\right)\right].
\end{split}
\end{equation}
\end{lemma}
\begin{proof} We split the proof into two cases.

{\bf Case 1: $t\le t_0(|x|)$.}\,\, We first consider the case $|x|\ge 2$. Let $U:=B(x,{|x|}/{2})$. By \eqref{e1-4},
\begin{equation}\label{l2-2-2}
\begin{split}
T_t^V 1(x)&=\Ee_x\left[\exp\left(-\int_0^t V(B_s)\,ds\right)\I_{\{\tau_U>t\}}\right]+
\Ee_x\left[\exp\left(-\int_0^t V(B_s)\,ds\right)\I_{\{\tau_U\le t\}}\right]\\
&=:I_1+I_2.
\end{split}
\end{equation}

It follows from \eqref{e1-1a} that
$\inf_{z\in U}V(z)\ge \frac{c_1}{(1+|x|)^\alpha}$,
which implies
\begin{equation*}
I_1\le \exp\left(-\frac{c_1 t}{(1+|x|)^\alpha}\right).
\end{equation*}
Meanwhile, we have
\begin{align*}
I_2&\le \Ee_x\left[\exp\left(-\int_0^t V(B_s)\,ds\right)\I_{\{\tau_U\le t,B_t\in B(x,{|x|}/{3})\}}\right] +\Ee_x\left[\exp\left(-\int_0^t V(B_s)\,ds\right)\I_{\{\tau_U\le t,B_t\notin B(x,{|x|}/{3})\}}\right]\\
&=:I_{21}+I_{22}.
\end{align*}
By the strong Markov property, it holds that
\begin{equation}\label{l2-2-3}
\begin{split}
I_{21}&\le \Pp_x\left(\tau_U\le t, B_t\in B(x,{|x|}/{3})\right)\le \Ee_x\left[\I_{\{\tau_U\le t\}}\Pp_{B_{\tau_U}}\left(B_{t-\tau_U}\in B(x,{|x|}/{3})\right)\right]\\
&\le \sup_{s\in [0,t],z\in \partial U}\int_{B(x,{|x|}/{3})}q(s,z,y)\,dy\\
&\le c_2\sup_{s\in [0,t]}\left(  \left(\frac{|x|^2}{s}\right)^{{d}/{2}}\exp\left(-\frac{c_3|x|^2}{s}\right)\right)
\le c_4\exp\left(-\frac{c_5|x|^2}{t}\right),
\end{split}
\end{equation}
where in the fourth inequality we have used \eqref{e1-6},
and the last inequality is due to the fact $t\le t_0(|x|)\le (2+|x|)^2\le 4|x|^2$ (thanks to
$|x|\ge 2$).
On the other hand,
\begin{equation}\label{l2-2-4}
\begin{split}
I_{22}&\le \Pp_x\left(B_t\notin B(x,{|x|}/{3})\right)=\int_{B(x,{|x|}/{3})^c}q(t,x,y)\,dy\\
&\le c_6t^{-{d}/{2}}\int_{{|x|}/{3}}^\infty e^{-\frac{c_7r^2}{t}}r^{d-1}\,dr=c_6\int_{{|x|}/{3}}^\infty \left(\frac{r^2}{t}\right)^{{d}/{2}}e^{-\frac{c_7r^2}{t}}r^{-1}\,dr
\\
&\le c_8\exp\left(-\frac{c_9|x|^2}{t}\right),
\end{split}
\end{equation}
where the second inequality is due to \eqref{e1-6}
and in the last inequality we have also used the fact $t\le 4|x|^2$ thanks to $|x|\ge 2$.

Putting both estimates above for $I_1$ and $I_2$ together yields that for all $t\le t_0(|x|)$,
\begin{align*}
T_t^V1(x)&\le c_{10}\left(\exp\left(-\frac{c_{11}t}{(1+|x|)^\alpha}\right)+\exp\left(-\frac{c_{11}|x|^2}{t}\right)\right)\le c_{12}\exp\left(-\frac{c_{13}t}{(1+|x|)^\alpha}\right),
\end{align*}
where we have used the fact that
\begin{align*}
\frac{t}{(1+|x|)^\alpha}\le \frac{c_{14}|x|^2}{t},\quad t\le t_0(|x|),\ |x|\ge 2.
\end{align*}

When $|x|\le 2$, $t\le t_0(|x|)\le c_{15}$ and so
\begin{align*}
T_t^V1(x)&\le c_{16}\le c_{17}\exp\left(-\frac{c_{18}t}{(1+|x|)^\alpha}\right).
\end{align*}

Hence, we have proved \eqref{l2-2-1} for the case that $|x|\le t_0(|x|)$.

{\bf Case 2: $t>t_0(|x|)$.}\,\, If $t>t_0(|x|)$, then  $|x|\le c_1t^{{2}/({2+\alpha})}$. Let $R(t):=2c_1t^{{2}/({2+\alpha})}$ and
$U:=B(x,R(t))$. We still define $I_1$ and $I_2$ by \eqref{l2-2-2} with $U=B(x,R(t))$.
By some direct computations, we obtain
\begin{align*}
I_1&\le \exp\left(-t\cdot \inf_{z\in U}V(z)\right)\le
\exp\left(-c_2t\cdot \frac{1}{(1+R(t))^\alpha}\right)\le \exp\left(-c_3t^{(2-\alpha)/(2+\alpha)}\right),
\end{align*}
where in the last inequality we have used the fact that $t\ge t_0(|x|)\ge 1$.
Meanwhile, it holds that
\begin{align*}
I_2&\le \Ee_x\left[\exp\left(-\int_0^t V(B_s)\,ds\right)\I_{\{\tau_U\le t,B_t\in B(x,{R(t)}/{3})\}}\right]\\
&\quad+\Ee_x\left[\exp\left(-\int_0^t V(B_s)\,ds\right)\I_{\{\tau_U\le t,B_t\notin B(x,{R(t)}/{3})\}}\right]\\
&=:I_{21}+I_{22}.
\end{align*}
Applying the same arguments for \eqref{l2-2-3} and \eqref{l2-2-4} with $B(x,{|x|}/{3})$ replaced by $B(x,{R(t)}/{3})$, and noting that $R(t)\ge c_4t_0(|x|)^{{2}/({2+\alpha})}\ge c_5$, we can prove
\begin{align*}
I_{21}+I_{22}&\le c_6\exp\left(-\frac{c_7R(t)^2}{t}\right)\le c_8\exp\left(-c_9t^{(2-\alpha)/(2+\alpha)}\right).
\end{align*}
Combining with all the estimates above, we can deduce that for all $t> t_0(|x|)$,
\begin{align*}
T_t^V1(x)&\le c_{10}\exp\left(-c_{11}t^{(2-\alpha)/(2+\alpha)}\right).
\end{align*} The proof is complete.
\end{proof}

\begin{remark} \label{remark:one} Explicit bounds of $T_t^V1(x)$ somehow indicate the relation between the Schr\"{o}dinger heat kernel $p(t,x,y)$ and the Gaussian heat kernel $q(t,x,y)$. In particular, according to Theorem \ref{t1-1}, the right hand side of \eqref{l2-2-1} is the multiple weighted function of $q(t,x,y)$ for $p(t,x,y)$. \end{remark}

\subsection{Proof of Theorem \ref{t1-1}: the case that
$0<t\le C_0t_0(\max\{|x|,|y|\})$
} From this subsection, without any mention we always assume that $|x|\le |y|$. We firstly prove the following
on-diagonal estimate.
\begin{lemma}\label{l3-1}
For any positive constants $C_0$ and $C_0'$, there exist positive constants $C_{1}$, $C_{2}$, $C_{3}$ and $C_{4}$ such that for all $t>0$ and $x,y\in \R^d$ with $0<t\le C_0t_0(|y|)$ and $|x-y|\le C_0't^{{1}/{2}}$,
\begin{equation}\label{l3-1-1}
\begin{split}
&C_{1}t^{- {d}/{2}}\exp\left(-\frac{C_{2}t}{(1+|y|)^\alpha}\right)\le p(t,x,y)\le
C_{3}t^{- {d}/{2}}\exp\left(-\frac{C_{4}t}{(1+|y|)^\alpha}\right).
\end{split}
\end{equation}
\end{lemma}
\begin{proof}
Suppose that $0<t\le C_0t_0(|y|)$ and $|x-y|\le C_0't^{{1}/{2}}$. By the semigroup property, we have
\begin{align*}
p(t,x,y)&=\int_{\R^d}p\left({t}/{2},y,z\right)p\left({t}/{2},z,x\right)\,dz\le c_1t^{- {d}/{2}}\int_{\R^d}p\left({t}/{2},y,z\right)\,dz=
c_1t^{- {d}/{2}}T_{{t}/{2}}^V 1(y)\\
&\le c_2t^{-{d}/{2}}\exp\left(-\frac{c_3t}{(1+|y|)^\alpha}\right).
\end{align*}
Here in the first inequality we have used the fact
\begin{align*}
p\left({t}/{2},z,x\right)\le q\left({t}/{2},z,x\right)\le c_1t^{-{d}/{2}},
\end{align*}
and the last inequality follows from that for all $t\le C_0t_0(|y|)$,
\begin{align*}
T_{{t}/{2}}^V 1(y)&\le c_4\left(\exp\left(-\frac{c_5t}{(1+|y|)^\alpha}\right)+\exp\left(-c_5t^{(2-\alpha)/(2+\alpha)}\right)\right)\le c_6\exp\left(-\frac{c_7t}{(1+|y|)^\alpha}\right),
\end{align*}
where we have used \eqref{l2-2-1} and the fact that
$$\frac{t}{(1+|y|)^\alpha}\le c_8t^{(2-\alpha)/(2+\alpha)},\quad t\le C_0t_0(|y|).$$

Next, define $U:=B(y,2C_0't^{{1}/{2}})$. Since
$$t^{{1}/{2}}\le C_0^{{1}/{2}}t_0(|y|)^{{1}/{2}}\le C_0^{{1}/{2}}(1+|y|)^{1/2+{\alpha}/{4}},$$
we have $\sup_{z\in U}V(z)\le c_9(1+|y|)^{-\alpha}$. Taking any $f\in C_b(\R^d)$ with ${\rm supp}[f]\subset B(y,C_0't^{{1}/{2}})$, it holds that
\begin{align*}
T_t^Vf(y)&=\Ee_y\left[f(B_t)\exp\left(-\int_0^t V(B_s)\,ds\right)\right]\ge \Ee_y\left[f(B_t)\exp\left(-\int_0^t V(B_s)\,ds\right)\I_{\{t<\tau_U\}}\right]\\
&\ge \exp\left(-t\sup_{z\in U}V(z)\right)\Ee_y\left[f(B_t)\I_{\{t<\tau_U\}}\right]\ge \exp\left(-\frac{c_{10}t}{(1+|y|)^\alpha}\right)\int_{B(y,C_0't^{{1}/{2}})}q_U(t,y,z)f(z)\,dz\\
&\ge c_{11}t^{-{d}/{2}}\exp\left(-\frac{c_{10}t}{(1+|y|)^\alpha}\right)\int_{B(y,C_0't^{{1}/{2}})}f(z)\,dz.
\end{align*}
Here the third inequality is due to the fact that ${\rm supp}[f]\subset B(y,C_0't^{{1}/{2}})$, and in the last inequality we have used
\begin{equation}\label{e:Diri}
q_{B(y,2C_0't^{{1}/{2}})}(t,z_1,z_2)\ge c_{11}t^{-{d}/{2}},\quad z_1,z_2\in B(y,C_0't^{{1}/{2}}),
\end{equation}
which is proved by \eqref{l2-1-2}.
According to the expression above, we can deduce the lower bound of $p(t,x,y)$ in \eqref{l3-1-1} immediately.
\end{proof}

\begin{lemma}\label{l3-3}
For any positive constants $C_0$ and $C_0'$, there exist positive constants $C_{5}$ and $C_{6}$ such that for all $t>0$ and $x,y\in \R^d$ with  $0<t\le C_0t_0(|y|)$ and $|x-y|> 2C_0't^{1/2}$,
\begin{equation}\label{l3-3-1}
 p(t,x,y)\le C_{5}t^{-d/2}\exp\left(-C_{6}\left(\frac{t}{(1+|y|)^\alpha}
 +\frac{|x-y|^2}{t}\right)\right).
\end{equation}
\end{lemma}
\begin{proof}
Suppose that $0<t\le C_0t_0(|y|)$ and $|x-y|> 2C_0't^{1/2}$.
By the semigroup property we have
\begin{align*}
  p(t,x,y)
&=\int_{\{z:|z-y|\le {|x-y|}/{2}\}}p\left(t/2,x,z\right)p\left(t/2,z,y\right)dz +
\int_{\{z:|z-y|> {|x-y|}/{2}\}}p\left(t/2,x,z\right)p\left(t/2,z,y\right)dz\\
&=:I_1+I_2.
\end{align*}

When $|z-y|\le {|x-y|}/{2}$, it holds that
\begin{equation*}
|z-x|\ge |x-y|-|z-y|\ge \frac{|x-y|}{2}\ge C_0't^{1/2}.
\end{equation*}
Then
\begin{align*}
p\left(t/2,x,z\right)&\le q(t/2,x,z)\le c_1t^{-d/2}\exp\left(-\frac{c_2|x-z|^2}{t}\right)\le c_3t^{-d/2}\exp\left(-\frac{c_4|x-y|^2}{t}\right).
\end{align*}
Hence, we obtain
\begin{align*}
I_1&\le \sup_{z\in \R^d:|z-y|\le {|x-y|}/{2}}p\left(t/2,x,z\right)\cdot \int_{\R^d}p\left(t/2,z,y\right)\,dz\le c_3t^{-d/2}\exp\left(-\frac{c_4|x-y|^2}{t}\right)\cdot T_{t/2}^V 1(y)\\
&\le c_5t^{-d/2}\exp\left(-c_6\left(\frac{|x-y|^2}{t}+\frac{t}{(1+|y|)^\alpha}\right)\right),
\end{align*}
where the third inequality follows from \eqref{l2-2-1} and the fact $0<t\le C_0t_0(|y|)$.

Similarly, when $|z-y|>{|x-y|}/{2}>C_0't^{1/2}$, we get
\begin{align*}
p\left(t/2, z,y\right)&\le  q\left(t/2,z,y\right)\le c_{7}t^{-d/2}\exp\left(-\frac{c_{8}|y-z|^2}{t}\right)\le  c_{9}t^{-d/2}\exp\left(-\frac{c_{10}|x-y|^2}{t}\right).
\end{align*}
Therefore, using \eqref{l2-2-1} again we have
\begin{align*}
I_2& \le \sup_{z\in \R^d:|z-y|> {|x-y|}/{2}}p\left(t/2,z,y\right)\cdot \int_{\R^d}p\left(t/2,x,z\right)\,dz\\
&\le c_{11}t^{-d/2}\exp\left(-\frac{c_{12}|x-y|^2}{t}\right)\cdot\left(\exp\left(-\frac{c_{12}t}{(1+|x|)^\alpha}\right)+\exp\left(-c_{12}t^{(2-\alpha)/(2+\alpha)}\right)\right)\\
&\le c_{11}t^{-d/2}\exp\left(-\frac{c_{12}|x-y|^2}{t}\right)\cdot\left(\exp\left(-\frac{c_{12}t}{(1+|y|)^\alpha}\right)+\exp\left(-c_{12}t^{(2-\alpha)/(2+\alpha)}\right)\right)\\
&\le c_{13}t^{-d/2}\exp\left(-\frac{c_{14}|x-y|^2}{t}\right)\cdot \exp\left(-\frac{c_{14}t}{(1+|y|)^\alpha}\right),
\end{align*}
where we have used the facts that $|x|\le |y|$ and $t\le C_0t_0(|y|)$.

Combining with both estimates for $I_1$ and $I_2$, we can prove \eqref{l3-3-1}.
\end{proof}

To consider off-diagonal lower bound estimates of $p(t,x,y)$ for $ 0<t\le C_0t_0(|y|)$, we further need to split it into two lemmas.

\begin{lemma}\label{l3-2}
Suppose that $d\ge 2$. For any positive constants $C_0$ and $C_0'$,  there exist positive constants $C_{7}$ and $C_{8}$ such that for all $0<t\le C_0t_0(|x|)$ and $|x-y|> 8C_0't^{{1}/{2}}$,
\begin{equation}\label{l3-2-1}
 p(t,x,y)\ge C_{7}t^{-{d}/{2}}\exp\left(-C_{8}\left(\frac{|x-y|^2}{t}+\frac{t}{(1+|y|)^\alpha}\right)\right).
\end{equation}
\end{lemma}
\begin{proof}
Without loss of generality we assume that $|x|\ge 2$; otherwise, the desired estimate \eqref{l3-2-1} automatically holds, by
noting that $t_0(|x|)\le 2$ for every $|x|\le 2$ and
$p(t,x,y)\ge q(t,x,y)e^{-c_0t}$ for all $x,y\in \R^d$ and $t>0$, thanks to
\eqref{note-}.
The proof is split into three cases.

 {\bf Case 1: $|x-y|>|y|/4$ and $|x|\le {|y|}/{4}$.}\,\, Set $y_0:=\frac{|y|}{|x|}\cdot x$.  One can find $\kappa_0\in (0,1/4)$ small enough such that
for every $y'\in D_0:=B(y_0,\kappa_0|y|)$
and $x'\in B(x,|x|/8)$,
\begin{align*}
|y'-x'|&\ge |y_0-x|-|y'-y_0|-|x'-x |=|y|-|x|-|y'-y_0|-|x'-x |\\
&\ge |y|/4\ge |x-y|/8\ge  C_0't^{1/2},
\end{align*} where in the second inequality we have used the assumption that
$|x|\le |y|/4$.

Now for every $y'\in D_0$, let $U:=B\left(y', |y'-x|-{C_0't^{1/2}}/{4}\right)$ and
$W:=B(x,C_0't^{1/2})$. Hence, choosing $\kappa_0$ small enough if necessary,
\begin{equation}\label{l3-2-2}
|\partial U \cap W|\ge c_1t^{(d-1)/2}.
\end{equation} Furthermore, according to the definitions of $y_0$ and $D_0$, there exists $x^*\in  B(x,|x|/8)$ so that
$y'=\frac{|y'|}{|x^*|}\cdot x^*$. Then, for all $z\in U$,
\begin{equation}\label{l3-2-1a}
\begin{split}
  |z|\ge &|y'|-|z-y'|\ge |y'|-|y'-x|\ge |y'|-|y'-x^*|-|x^*-x|\\
  \ge &|y'|-(|y'|-|x^*|)-\frac{|x|}{8}\ge |x^*|-\frac{|x|}{8}\ge \frac{3| x|}{4},
\end{split}
\end{equation} where in the fifth inequality we used the fact that $|y'|\ge \frac{3|y|}{4}\ge \frac{9}{8}|x|\ge |x^*|$ thanks to $|x|\le \frac{|y|}{4}$ and $\kappa_0\in (0,1/4)$.
In particular, by \eqref{l3-2-1a}, it is obvious to verify that
\begin{equation}\label{l3-2-2a}
\sup_{u\in U}V(u)\le c_2(1+|x|)^{-\alpha}.
\end{equation}
By \eqref{l2-4-2},  for every $y'\in D_0$,
\begin{equation}\label{l3-2-3}
\begin{split}
p(t,x,y')
&=\Ee_{y'}\left[\exp\left(-\int_0^{\tau_U}V(B_s)\,ds\right)p\left(t-\tau_U,B_{\tau_U},x\right)\I_{\{\tau_U\le t\}}\right]\\
&\ge \Ee_{y'}\left[\exp\left(-\tau_U \cdot \sup_{u\in U}V(u)\right)p\left(t-\tau_U,B_{\tau_U},x\right)
\I_{\{\tau_U\le t,B_{\tau_U}\in W\}}\right]\\
&=\frac 1{2}\int_0^t \exp\left(-s\sup_{u\in U}V(u)\right)\left(\int_{\partial U \cap W}p(t-s,z,x)\frac{\partial q_U(s,y',\cdot)}{\partial n}(z)\,\sigma(dz)\right)\,ds\\
&\ge \frac{c_3|x-y'|}{t^{1/2}}\int_0^{t/2} e^{-\frac{c_4s}{(1+|x|)^\alpha}}(t-s)^{-d/2}e^{-\frac{c_4(t-s)}{(1+|x|)^\alpha}}
\left(\frac{t}{s}\right)^{d/2}e^{-\frac{c_4|x-y'|^2}{s}}s^{-1}\,ds\\
&\ge c_5 t^{-d/2}e^{-\frac{c_6t}{(1+|x|)^\alpha}}\int_0^{t/2}\exp\left(-c_6\left(\frac{s}{(1+|x|)^\alpha}+\frac{|x-y|^2}{s}\right)\right)s^{-1}\,ds.
\end{split}
\end{equation}
Here
the second equality is due to \eqref{l2-5-1},
the second inequality follows from
the facts that
$|x-y'|\ge C_0't^{1/2}$,
\eqref{l3-2-2}, \eqref{l2-1-1} and
\begin{align*}
\sup_{z\in W}p(t-s,z,x)\ge c_7(t-s)^{-d/2}e^{-\frac{c_8(t-s)}{(1+|x|)^\alpha}},\quad 0<s\le t/2,
\end{align*}
which can be deduced from \eqref{l3-1-1}, and the last inequality follows from the
facts that $$4(2+\kappa_0)|x-y|\ge (2+\kappa_0)|y|\ge |x|+|y'|\ge |x-y'|\ge \left(\frac{1}{4}-\kappa_0\right)|y|\ge \frac{1}{4} \left(\frac{1}{4}-\kappa_0\right)|x-y|$$ and ${|x-y|}\ge 8C_0't^{{1}/{2}}. $
Meanwhile, we deduce that
\begin{align*}
&\int_0^{{t}/{2}}\exp\left(-c_6\left(\frac{s}{(1+|x|)^\alpha}+\frac{|x-y|^2}{s}\right)\right)s^{-1}\,ds\\
&\ge
\begin{cases}
\displaystyle \int_{{t}/{4}}^{{t}/{2}}\exp\left(-c_6\left(\frac{s}{(1+|x|)^\alpha}+\frac{|x-y|^2}{s}\right)\right)s^{-1}\,ds,\ & t\le 4|x-y|(1+|x|)^{\alpha/2},\\
\displaystyle\int_{|x-y|(1+|x|)^{\alpha/2}}^{2|x-y|(1+|x|)^{\alpha/2}}\exp\left(-c_6\left(\frac{s}{(1+|x|)^\alpha}+\frac{|x-y|^2}{s}\right)\right)s^{-1}\,ds,\ & t> 4|x-y|(1+|x|)^{\alpha/2},\\
\end{cases}
\\
&\ge
\begin{cases}
c_{9}\exp\left(-c_{10}\left(\frac{t}{(1+|x|)^\alpha}+\frac{|x-y|^2}{t}\right)\right),\ & t\le 4|x-y|(1+|x|)^{\alpha/2},\\
c_{9}\exp\left(-\frac{c_{10}|x-y|}{(1+|x|)^{\alpha/2}}\right),\ & t> 4|x-y|(1+|x|)^{\alpha/2},\\
\end{cases}
\\
&\ge c_{11} \exp\left(-c_{12}\left(\frac{t}{(1+|x|)^\alpha}+\frac{|x-y|}{(1+|x|)^{\alpha/2}}+\frac{|x-y|^2}{t}\right)\right).
\end{align*}
Putting such estimate into \eqref{l3-2-3}, we obtain
\begin{align*}
p(t,x,y')&\ge
 c_{13}t^{- {d}/{2}}\exp\left(-c_{14}\left(\frac{t}{(1+|x|)^\alpha}+\frac{|x-y|}{(1+|x|)^{\alpha/2}}+\frac{|x-y|^2}{t}\right)\right).
\end{align*}
This, along with the fact
\begin{align*}
\frac{|x-y|}{(1+|x|)^{\alpha/2}}\le c_{15}\left(\frac{|x-y|^2}{t}+\frac{t}{(1+|x|)^\alpha}\right),
\end{align*}
implies that for any $0<t\le C_0t_0(|x|)$ and $y'\in D_0$,
\begin{equation}\label{l3-2-4}
\begin{split}
p(t,x,y')&\ge c_{16}t^{-d/2}\exp\left(-c_{17}\left(\frac{t}{(1+|x|)^\alpha}+\frac{|x-y|^2}{t}\right)\right)\\
&\ge c_{16}t^{-d/2}\exp\left(- \frac{c_{18}|x-y|^2}{t} \right)\ge c_{16}t^{-d/2}\exp\left(-c_{18}\left(\frac{t}{(1+|y|)^\alpha}+\frac{|x-y|^2}{t}\right)\right),
\end{split}
\end{equation} where in the second inequality we used the fact that $$\frac{|x-y|^2}{t}\ge \frac{c_{19}|y|^2}{t}\ge \frac{c_{19}|x|^2}{t}\ge \frac{c_{20}t}{(1+|x|)^\alpha}$$ since $|y-x|\ge |y|/4\ge |x|/4$ and $0<t\le C_0t_0(|x|)$.

Note that $d\ge 2$. According to the set $D_0=B(y_0,\kappa_0|y|)$ constructed above,
we can find a series of balls $D_k=B(y_i,\kappa_0|y|)$, $1\le k \le 1+N(y)$, such that
the following statements hold:
\begin{itemize}
\item [(i)] $y\in D_{1+N(y)}$;

\item [(ii)]
\begin{equation}\label{l3-2-5}
\sup_{y\in \R^d} N(y)\le c_{21};
\end{equation}

\item [(iii)]
\begin{equation}\label{l3-2-5a}
c_{22}|y|\ge |z_k-z_{k+1}|\ge c_{23}|y|\ge \frac{c_{23}|x-y|}{2} \hbox{ for all } z_k\in D_k \hbox{ and } 0\le k \le N(y);
\end{equation}

\item [(iv)] There exists $\kappa_1>0$ such that for every $0\le k \le N(y)$ we can find $y_k'\in \R^d$ so that
\begin{equation}\label{l3-2-6}
\begin{split}
 D_k\cup D_{k+1}\subset B\left(y_k',\kappa_1|y|\right)\hbox{ and }
 |z|\ge
 4c_{22}
 |y|\hbox{ for all } z\in  B\left(y_k',\kappa_1|y|\right),
\end{split}
\end{equation}
where $c_{22}$ is the constant given in \eqref{l3-2-5a} (that can be achieved by taking $\kappa_0$ small enough).
\end{itemize}

According to
\eqref{l3-2-5a} and \eqref{l3-2-6}, we know that $t \le C_0t_0(|x|)\le c_{24}t_0(|z_k|)$ and
$|z_k-z_{k+1}|\ge 2c_{25}t^{1/2}$ for all $z_k\in D_k$ and $0\le k \le N(y)$. Hence, by \eqref{l3-2-6} (which ensures
that $|z|\ge c_{26}|y|$ for every $z\in B\left(z_k,|z_k-z_{k+1}|-c_{25}t^{1/2}\right)$), we can apply the  argument of
\eqref{l3-2-4}
(with $U=B\left(z_k,|z_k-z_{k+1}|-c_{25}t^{1/2}\right)$ and $W=B\left(z_{k+1},2c_{25}t^{1/2}\right)$)
to obtain that for every $0< t \le C_0t_0(|x|)$, $ z_k\in D_k$ and $0\le k \le N(y)$,
\begin{equation}\label{l3-2-7}
\begin{split}
&  p\left(\frac{t}{2(N(y)+1)},z_k,z_{k+1}\right)\\
&\ge c_{27}\left(\frac{t}{2(N(y)+1)}\right)^{-d/2}
\exp\left(-c_{28}\left( \frac{2(N(y)+1)|z_k-z_{k+1}|^2}{t}+\frac{t}{2N(y)(1+|y|)^\alpha}\right)\right)\\
&\ge c_{29}t^{-d/2}\exp\left(-c_{30}\left( \frac{|x-y|^2}{t}+\frac{t}{(1+|y|)^\alpha}\right)\right),
\end{split}
\end{equation}
where the second inequality follows from  \eqref{l3-2-5}, \eqref{l3-2-5a} and the fact that $|x-y|\ge |y|/4$.

Combining this with \eqref{l3-2-4} yields that
\begin{align*}
& p(t,y,x)\\
&\ge \int_{D_{N(y)}}\cdots \int_{D_0}p\left(\frac{t}{2(N(y)+1)},y,z_{N(y)}\right)\left(\prod_{i=1}^{N(y)}
p\left(\frac{t}{2(N(y)+1)},z_{i},z_{i-1}\right)\right)p(t/2,z_0,x)\prod_{i=0}^{N(y)}dz_i\\
&\ge c_{31}\left(\prod_{i=0}^{N(y)}|D_i|\right)\cdot \left(t^{-d/2}\exp\left(-c_{32}\left(\frac{|x-y|^2}{t}+\frac{t}{(1+|y|)^\alpha}\right)\right)\right)^{N(y)+2} \\
&\ge  c_{33} t^{-d/2}
\left(\frac{|y|^2}{t}\right)^{\frac{d(N(y)+1)}{2}}\cdot \exp\left(-c_{34}(N(y)+2)\left(\frac{|x-y|^2}{t}+\frac{t}{(1+|y|)^\alpha}\right)\right)\\
&\ge c_{35}t^{-d/2}\exp\left(-c_{36}\left(\frac{|x-y|^2}{t}+\frac{t}{(1+|y|)^\alpha}\right)\right),
\end{align*}
where the last inequality follows from \eqref{l3-2-5} and the facts that $|y|\ge |x-y|/2\ge 4C_0't^{1/2}$ and $|y|\ge |x|\ge 2$. Therefore, we have finished
the proof for {\bf Case 1}.

{\bf Case 2: $|x-y|>|y|/4$ and $|x|>|y|/4$.}\,\, As explained in {\bf Case 1}, since $d\ge 2$, for every $|x-y|>|y|/4$ and $|x|>|y|/4$, we can find
a series of balls $D_k$, $0\le k \le N(y)+1$, such that $x\in D_0$, $y\in D_{N(y)+1}$  and the properties
\eqref{l3-2-5}--\eqref{l3-2-6} hold.
Using them and repeating the procedure in {\bf Case 1}, we know
\eqref{l3-2-7} still holds. Hence,
\begin{align*}
& p(t,y,x)\\
&\ge \int_{D_{N(y)}}\cdots \int_{D_1}p\left(\frac{t}{ N(y)+1 },y,z_{N(y)}\right)\left(\prod_{i=2}^{N(y)}
p\left(\frac{t}{N(y)+1},z_{i},z_{i-1}\right)\right)p\left(\frac{t}{N(y)+1},z_1,x\right)\prod_{i=1}^{N(y)}dz_i\\
&\ge c_{1}\left(\prod_{i=1}^{N(y)}|D_i|\right)\cdot \left(t^{-d/2}\exp\left(-c_{2}\left(-\frac{|x-y|^2}{t}+\frac{t}{(1+|y|)^\alpha}\right)\right)\right)^{N(y)+1} \\
&\ge  c_{3}t^{-d/2} \left(\frac{|y|^2}{t}\right)^{\frac{dN(y)}{2}}\cdot \exp\left(-c_{4}(N(y)+1)\left(-\frac{|x-y|^2}{t}+\frac{t}{(1+|y|)^\alpha}\right)\right)\\
&\ge c_{5}t^{-d/2}\exp\left(-c_{6}\left(-\frac{|x-y|^2}{t}+\frac{t}{(1+|y|)^\alpha}\right)\right).
\end{align*}
Here in the last inequality we have also used the facts that $N(y)\le c_{7}$ and $|y|\ge |x-y|/2\ge 4C_0't^{1/2}$. Thus, we prove the assertion for {\bf Case 2}.

{\bf Case 3: $|x-y|\le |y|/4$.}\,\, In this case, we just take $U:=B\left(y,|x-y|-C_0't^{1/2}\right)$ and $W:=B\left(x,2C_0't^{1/2}\right)$.
It is easy to see that
$$
|z|\ge c_1|y|,\quad z\in U.
$$
Applying this and the argument of \eqref{l3-2-4}, we can prove the desired conclusion.
\end{proof}

\begin{lemma}\label{l3-4}
Suppose that $d\ge 2$. Then, for any $C_0'>0$,  there exist constants $C_0$,  $C_9$ and $C_{10}>0$ so that for all $x,y\in \R^d$ and  $C_0t_0(|x|)<t\le C_0't_0(|y|)$,
\begin{equation}\label{l3-4-1}
 p(t,x,y)\ge C_{9}t^{-{d}/{2}}\exp\left(-C_{10}\left(\frac{|x-y|^2}{t}+\frac{t}{(1+|y|)^\alpha}\right)\right).
\end{equation}
\end{lemma}
\begin{proof}
Without loss of generality we assume that $t\ge 2$, $C_0't_0(|y|)\ge 4C_0t_0(|x|)$ and $|y|\ge 4|x|$; otherwise, by \eqref{e1-4} and $t_0(|x|)\simeq t_0(|y|)$, the desired conclusion holds true, thanks to \eqref{l3-2-1}.

Below, we let $x\neq 0$ first. Note that
$t>C_0t_0(|x|)=C_0(1+|x|)^{{(\alpha+2)}/{2}}$, so $|x|\le \left({t}/{C_0}\right)^{{2}/({2+\alpha})}$.
For any $\delta_0>0$, define
\begin{align*}
x_0=x,\ R_0=1+\frac{|x|}{2},\ x_k:=\left(1+\sum_{i=0}^{k-1}R_i\right) \frac{x}{|x|},\ R_k:=(1+\delta_0)R_{k-1},\quad 1\le k \le N(x),
\end{align*}
where
$N(x)$ is the smallest positive integer (which may depend on $x$)
such that $|x_{N(x)}|\ge 2\left({t}/{C_0}\right)^{{2}/({2+\alpha})}$.
It is not difficult
to verify that
\begin{equation}\label{l3-4-2}
N(x)\le c_1\log (1+t),\ R_{N(x)}\simeq t^{2/(2+\alpha)},\quad \ x\in \R^d,
\end{equation} and, by taking $C_0>0$ large enough we have $$2\sum_{i=0}^{N(x)-1}R_i^{1+\alpha/2}\le t/2.$$
Furthermore, note that
\begin{align*}
\sum_{i=0}^{k-1}R_i\ge \frac{c_2}{\delta_0}\left(1-\frac{1}{(1+\delta_0)^k}\right)R_k,\quad k\ge 1,
\end{align*}
which ensures that, by choosing $\delta_0$ small enough, it holds that $|z|\ge c_3R_k$ for every $z\in B(x_k,R_k)$ and $k\ge k_0$ with some $k_0>0$
(also thanks to $\lim_{\delta_0 \downarrow 0}\delta_0^{-1}{\left(1-{(1+\delta_0)^{-k}}\right)}{}=k$).
Then, there exist positive constants $\delta_0$, $\kappa_0$ small enough such that for $D_k:=B\left(x_{k},\kappa_0R_{k}\right)$,
the following properties hold:
\begin{itemize}
\item [(i)]
\begin{equation}\label{l3-4-3}
|\partial B(z,R_{k})\cap D_{k+1}|\ge c_4R_{k}^{d-1},\quad z\in D_k,\ 1\le k \le N(x);
\end{equation}

\item [(ii)]
\begin{equation}\label{l3-4-4}
\sup_{u\in B(z,R_k)}V(u)\le c_5 R_k^{-\alpha},\quad\ z\in D_k,\, 1\le k \le N(x).
\end{equation}
\end{itemize}

Let $t_0:=t$ and $Q_0(t,y):=p(t,x,y)$. For each $1\le k \le N(x)$, define $$t_k:=t-2\sum_{i=0}^{k-1}R_i^{1+{\alpha}/{2}},\quad Q_k(t,y):=\inf_{s\in [t_k,t],z\in D_k}p(s,z,y)$$
Then, according to \eqref{l2-4-2},
for each $1\le k \le N(x)-1$, $t_k\le r \le t$ and $u\in D_k$, setting $U = B(u,R_k)$,
\begin{align*}
& p(r,u,y)\\
&\ge \Ee_u\left[\exp\left(-\tau_U \sup_{z\in U}V(z)\right)p\left(r-\tau_U,B_{\tau_U},y\right)\right]\\
&=\frac 1{2}\int_0^r \exp\left(-s \sup_{z\in U}V(z)\right)\cdot\left(\int_{\partial U}p(r-s,z,y)\frac{\partial q_U(s,u,\cdot)}{\partial n}(z)\,\sigma(dz)\right)\,ds\\
&\ge \frac 1{2}\int_0^r \exp\left(-\frac{c_6s}{R_k^\alpha} \right)\cdot\left(\int_{\partial U\cap D_{k+1}}p(r-s,z,y)\frac{\partial q_U(s,u,\cdot)}{\partial n}(z)\,\sigma(dz)\right)\,ds\\
&\ge c_7\bigg(\int_{R_k^{1+{\alpha}/{2}}}^{2R_k^{1+{\alpha}/{2}}} \exp\left(-\frac{c_8s}{R_k^\alpha} \right)
\frac{R_k^{d-1} R_k}{s^{d/2+1}}   \exp\left(-c_8\left(\frac{R_k^2}{s}+\frac{s}{R_k^2}\right)\right)\,ds\bigg)\cdot \inf_{s\in [r-2R_k^{1+\alpha/{2}},r],z\in D_{k+1}}p(s,z,y)\\
&\ge c_9Q_{k+1}(t,y)\int_{R_k^{1+{\alpha}/{2}}}^{2R_k^{1+{\alpha}/{2}}}
\exp\left(-c_{10}\left(\frac{R_k^2}{s}+\frac{s}{R_k^2}+\frac{s}{R_k^\alpha}\right)\right)\left(\frac{R_k^2}{s}\right)^{d/2}s^{-1}\,ds\\
&\ge c_{11}Q_{k+1}(t,y)\exp\left(-c_{12}R_k^{1-\alpha/{2}}\right).
\end{align*}
Here the second inequality is due to \eqref{l3-4-4}, the third inequality follows from \eqref{l2-1-1}, \eqref{l3-4-3}
and the fact $t_k\ge 2R_k^{1+\alpha/{2}}$, and
in the fourth inequality we have used the fact $r-2R_k^{1+\alpha/{2}}\ge t_{k+1}$ for every $r\ge t_k$.
Hence,
\begin{align}\label{l3-4-5}
Q_k(t,y)&=\inf_{s\in [t_k,t],u\in D_{k}}p(r,u,y)\ge c_{11}\exp\left(-c_{12}R_k^{1-\alpha/{2}}\right)Q_{k+1}(t,y),
\end{align}
where the constants $c_{11}$ and $c_{12}$ are independent of $k$.

Note that, by the definitions of $R_k$ and $N(x)$, we have
$t_{N(x)}=t-2\sum_{i=0}^{N(x)-1} R_i^{1+\alpha/{2}}\ge t/2$, and so
\begin{align*}
c_{13}t^{{2}/({\alpha+2})}\le |z|\le c_{14}t^{{2}/({\alpha+2})},\quad z\in D_{N(x)}=B\left(x_{N(x)},\kappa_0R_{N(x)}\right).
\end{align*}
This, along with the facts that $t\le C_0't_0(|y|)=C_0'(1+|y|)^{1+\alpha/{2}}$ and $|y|\ge 4|x|$,
yields that
$$t\le c_{15}t_0(|z|),\ |y-z|\le c_{16}|y|\le c_{17}|x-y|,\quad z\in D_{N(x)}.$$
Therefore, according to this, \eqref{l3-1-1} and \eqref{l3-2-1},
\begin{align*}
\inf_{s\in [t/2,t],z\in D_{N(x)}}p(s,z,y)&\ge c_{18}t^{-d/2}\exp\left(-c_{19}\left(\frac{|x-y|^2}{t}+\frac{t}{(1+|y|)^\alpha}\right)\right).
\end{align*}

Putting this into \eqref{l3-4-5} yields that
\begin{align*}
p(t,x,y)
&=Q_0(t,y)\ge c_{11}^{N(x)}\exp\left(-c_{12}\sum_{k=0}^{N(x)-1}R_k^{1-\alpha/{2}}\right)\inf_{s\in [t_{N(x)},t],z\in D_{N(x)}}p(s,z,y)\\
&\ge c_{20}t^{-d/2}\exp\left(-c_{21}\left(N(x)+R_{N(x)}^{1-\alpha/{2}}\right)\right)\exp\left(-c_{21}\left(\frac{|x-y|^2}{t}+\frac{t}{(1+|y|)^\alpha}\right)\right)\\
&\ge c_{22}t^{-d/2}\exp\left(-c_{23}t^{({2-\alpha})/({2+\alpha)}}\right)\exp\left(-c_{23}\left(\frac{|x-y|^2}{t}+\frac{t}{(1+|y|)^\alpha}\right)\right)\\
&\ge c_{24}t^{-d/2}\exp\left(-c_{25}\left(\frac{|x-y|^2}{t}+\frac{t}{(1+|y|)^\alpha}\right)\right).
\end{align*}
Here in the second inequality we have used
\begin{align*}
\sum_{k=0}^{N(x)-1}R_k^{1-\alpha/{2}}\le c_{26}R_{N(x)}^{1-\alpha/{2}},
\end{align*}
which can be verified by the facts that $R_{k+1}=(1+\delta_0)R_k$ and $\alpha\in (0,2)$, the third inequality is due to \eqref{l3-4-2}, and in the last inequality we have used
\begin{align*}
t^{(2-\alpha)/(2+\alpha)}\le \frac{c_{28}|y|^2}{t}\le \frac{c_{29}|x-y|^2}{t},\quad 0<t\le C_0't_0(|y|),\ |y|\ge 4|x|.
\end{align*}
The proof of the desired conclusion \eqref{l3-4-1} when $x\neq 0$ is finished.

Furthermore, for any $y\in \R^d$ and  $C_0t_0(0)<t\le C_0't_0(|y|)$ with $C_0$ large enough, it holds that
\begin{align*}p(t,0,y)\ge \int_{\{1\le |z|\le 2\}} p(1,0,z)p(t-1,z,y)\,dy.\end{align*} This along with the assertion above and Lemma \ref{l3-1} yields the desired conclusion \eqref{l3-4-1} when $x=0$. Then the proof is completed.
\end{proof}

Summarising Lemma \ref{l3-1}--Lemma \ref{l3-4}, we have the following two-sided estimates for $p(t,x,y)$ when $0<t\le C_0t_0(|y|)$.
\begin{proposition}\label{p3-1}
For every $C_0>0$, there exists $C_i>0$, $i=11,12,13,14$, so that for all $x,y\in \R^d$ with $|x|\le |y|$ and all $0<t\le C_0t_0(|y|)$,
\begin{equation}\label{p3-1-1}
\begin{split}
& p(t,x,y)\ge C_{11}t^{-d/2}\exp\left(-C_{12}\left(\frac{t}{(1+|y|)^\alpha}+\frac{|x-y|^2}{t}\right)\right),\\
& p(t,x,y)\le C_{13}t^{-d/2}\exp\left(-C_{14}\left(\frac{t}{(1+|y|)^\alpha}+\frac{|x-y|^2}{t}\right)\right).
\end{split}
\end{equation}
\end{proposition}

\subsection{Proof of Theorem \ref{t1-1}: the case that
$t>C_0t_0(\max\{|x|,|y|\})$
} We still assume that $|x|\le |y|$ in this part. The main result of this part is as follows.
\begin{proposition}\label{p3-2}
For every $C_0>0$, there exist constants $C_i$, $i=1,2,3,4$, so that for all $x,y\in \R^d$ with $|x|\le |y|$ and $t>C_0t_0(|y|)$,
\begin{equation}\label{p3-2-1}
\begin{split}
C_1\exp\left(-C_2t^{(2-\alpha)/(2+\alpha)}\right)\le p(t,x,y)\le
C_3\exp\left(-C_4t^{(2-\alpha)/(2+\alpha)}\right).
\end{split}
\end{equation}
\end{proposition}
\begin{proof}
First, for all $x,y\in \R^d$ with $|x|\le |y|$ and $t>C_0t_0(|y|)$,
\begin{align*}
p(t,x,y)&=\int_{\R^d}p\left(t/2,y,z\right)p\left(t/2,z,x\right)\,dz\le \sup_{z\in \R^d}p\left(t/2,z,x\right)\cdot T^V_{t/2}1(y)\\
&\le c_1t^{-d/2}\exp\left(-c_2t^{(2-\alpha)/(2+\alpha)}\right)\le c_3\exp\left(-c_4t^{(2-\alpha)/(2+\alpha)}\right),
\end{align*}
where the second inequality follows from \eqref{l2-2-1} and the fact $t>C_0t_0(|y|)$, and in the last inequality  we have
used the fact $t>C_0t_0(|y|)\ge c_5$.

Since $t>C_0t_0(|y|)\ge C_0t_0(|x|)$
(in particular, $|x|\le |y|\le c_6 t^{{2}/({\alpha+2})}$$)$, we can find a ball $B(u,1)\subset \R^d$ such that
\begin{equation}\label{p3-2-2}
c_7t^{{2}/({\alpha+2})}\le |z| \le c_8t^{{2}/({\alpha+2})},\quad
\min\{|z-x|,|z-y|\}\ge c_7t^{{2}/({\alpha+2})},\quad z\in B(u,1),
\end{equation}
which implies that
\begin{align*}
t\le c_9t_0(|z|),\quad z\in B(u,1).
\end{align*}
Combining this with the fact $t>C_0t_0(|y|)\ge C_0t_0(|x|)$, we can apply \eqref{l3-4-1} to derive that for all $z\in B(u,1)$,
\begin{align*}
p\left(t/2,x,z\right)&\ge c_{10}t^{-d/2}\exp\left(-c_{11}\left(\frac{|x-z|^2}{t}+\frac{t}{(1+|z|)^\alpha}\right)\right)\ge c_{12}\exp\left(-c_{13}t^{(2-\alpha)/(2+\alpha)}\right)
\end{align*} and
\begin{align*}
p\left(t/2,z,y\right)&\ge c_{10}t^{-d/2}\exp\left(-c_{11}\left(\frac{|y-z|^2}{t}+\frac{t}{(1+|z|)^\alpha}\right)\right)\ge c_{12}\exp\left(-c_{13}t^{(2-\alpha)/(2+\alpha)}\right).
\end{align*}
Here the last step in both inequalities above follows from the fact that $$|x-z|+|y-z|\le c_{14}(|x|+|y|+|z|)\le c_{15} t^{{2}/({\alpha+2})},$$ thanks to the first inequality in \eqref{p3-2-2}. Hence, it holds that
\begin{align*}
p(t,x,y)&\ge \int_{B(u,1)}p\left(t/2,x,z\right)p\left(t/2,z,y\right)\,dz
\ge c_{16}\exp\left(-c_{17}t^{(2-\alpha)/(2+\alpha)}\right).
\end{align*}
Therefore, we have finished the proof for \eqref{p3-2-1}.
\end{proof}

Finally, we present the

\begin{proof}[Proof of Theorem $\ref{t1-1}$]
Note that for all $x,y\in \R^d$ with $|x|\le |y|$ and for all $t>C_0t_0(|y|)$,
\begin{align*}
\frac{c_1|x-y|^2}{t}\le \frac{4c_1|y|^2}{t}\le c_2t^{(2-\alpha)/(2+\alpha)}.
\end{align*}
Then, Theorem \ref{t1-1} immediately follows from Propositions \ref{p3-1} and \ref{p3-2}.
\end{proof}

\subsection{Proof of Theorem \ref{t1-3}: estimates of Green function}

\begin{proof}[Proof of Theorem $\ref{t1-3}$]
Without loss of generality, throughout the proof we always assume that $|x|\le |y|$. The proof is split into two cases.

{\bf Case 1:\,\, $|x-y|\ge 1+|y|$.}\,\, In this case, it is easy to see that
$$
\frac{|x-y|^2}{t}\ge \frac{t}{(1+|y|)^{\alpha}},\quad 0<t\le (1+|y|)^{1+\alpha/2},
$$
which along with \eqref{t1-1-1} yields that \begin{equation}\label{t1-3-2}
p(t,x,y)\asymp
\begin{cases}
t^{-d/2}\exp\left(-\frac{|x-y|^2}{t}\right),\ \ &0<t\le (1+|y|)^{1+\alpha/2},\\
\exp\left(-t^{(2-\alpha)/(2+\alpha)}\right),\ &t>(1+|y|)^{1+\alpha/2}.
\end{cases}
\end{equation}
Hence,
\begin{align*}
G(x,y)&=\int_0^\infty p(t,x,y)\,dt\\
&\le c_1\int_0^{(1+|y|)^{1+\alpha/2}}t^{-d/2}\exp\left(-\frac{c_2|x-y|^2}{t}\right)\,dt+c_1
\int_{(1+|y|)^{1+\alpha/2}}^\infty \exp\left(-c_2t^{(2-\alpha)/(2+\alpha)}\right)\,dt\\
&=:I_1+I_2.
\end{align*}

By the change of variable $s=\frac{t}{|x-y|^2}$,  we have
\begin{align*}
I_1&= c_1|x-y|^{-(d-2)}\int_0^{\frac{(1+|y|)^{1+\alpha/2}}{|x-y|^2}}s^{-d/2}e^{-{c_2}/{s}}\,ds\\
&\le c_3|x-y|^{-(d-2)}\exp\left(-\frac{c_4|x-y|^2}{(1+|y|)^{1+\alpha/2}}\right)\le c_3|x-y|^{-(d-2)}\exp\left(-\frac{c_4|x-y|}{(1+|y|)^{\alpha/2}}\right),
\end{align*}
where in the first inequality we have used the fact
\begin{align*}
\int_0^{\frac{(1+|y|)^{1+\alpha/2}}{|x-y|^2}}s^{-d/2}e^{-{c_2}/{s}}\,ds\le c_3
 \exp\left(-\frac{c_4|x-y|^2}{(1+|y|)^{1+\alpha/2}}\right),
\end{align*}
thanks to  $\frac{(1+|y|)^{1+\alpha/2}}{|x-y|^2}\le 1$, and
the last inequality follows from $|x-y|\ge 1+|y|$.

Meanwhile, it holds that
\begin{align*}
I_2&\le c_5\exp\left(-c_6\left((1+|y|)^{1+\alpha/2}\right)^{(2-\alpha)/(2+\alpha)}\right)= c_5
\exp\left(-c_6(1+|y|)^{1-\alpha/2}\right)\\
&\le c_7|x-y|^{-(d-2)}\exp\left(-\frac{c_8|x-y|}{(1+|y|)^{\alpha/2}}\right),
\end{align*}
where the last inequality follows from the facts that $|x-y|\simeq (1+|y|)$ and $|y|\ge |x|$.

Putting both estimates for $I_1$ and $I_2$ together and using the fact $|x-y|\simeq 1+|y|$ again, we obtain that
$$
G(x,y)\le c_{9}|x-y|^{-(d-2)}\exp\left(-\frac{c_{10}|x-y|}{(1+|y|)^{\alpha/2}}\right).
$$

On the other hand, according to \eqref{t1-3-2}, we have
\begin{align*}
G(x,y)&\ge c_{11}\int_{(1+|y|)^{1+\alpha/2}}^\infty \exp\left(-c_{12}t^{(2-\alpha)/(2+\alpha)}\right)\,dt\ge c_{13}\exp\left(-c_{14}(1+|y|)^{1-\alpha/2}\right)\\
&\ge c_{15}|x-y|^{-(d-2)}\exp\left(-\frac{c_{16}|x-y|}{(1+|y|)^{\alpha/2}}\right),
\end{align*}
where in the last inequality we have used again the fact that $|x-y|\simeq (1+|y|)$.

{\bf Case 2:\,\, $|x-y|\le 1+|y|$.}\,\, Note that
\begin{align*}
&\frac{|x-y|^2}{t}\ge \frac{t}{(1+|y|)^{\alpha}},\quad 0<t\le |x-y|(1+|y|)^{\alpha/2},\\
&\frac{|x-y|^2}{t}\le \frac{t}{(1+|y|)^{\alpha}},\quad |x-y|(1+|y|)^{\alpha/2}<t\le (1+|y|)^{1+\alpha}.
\end{align*}
By \eqref{t1-1-1}, it holds that
\begin{equation}\label{t1-3-3}
p(t,x,y)\asymp
\begin{cases}
t^{-d/2}\exp\left(-\frac{|x-y|^2}{t}\right),\ \ &0<t\le |x-y|(1+|y|)^{\alpha/2},\\
t^{-d/2}\exp\left(-\frac{t}{(1+|y|)^\alpha}\right),\ \  &|x-y|(1+|y|)^{\alpha/2}<t \le (1+|y|)^{1+\alpha/2},\\
\exp\left(-t^{(2-\alpha)/(2+\alpha)}\right),\ &t>(1+|y|)^{1+\alpha/2}.
\end{cases}
\end{equation}
Then,
\begin{align*}
G(x,y)&=\int_0^\infty p(t,x,y)\,dt\\
&\le c_1\int_0^{|x-y|(1+|y|)^{\alpha/2}}t^{-d/2}\exp\left(-\frac{c_2|x-y|^2}{t}\right)\,dt+c_1
\int_{|x-y|(1+|y|)^{\alpha/2}}^{(1+|y|)^{1+\alpha/2}}t^{-d/2}\exp\left(-\frac{c_2 t}{(1+|y|)^\alpha}\right)\,dt\\
&\quad +c_1\int_{(1+|y|)^{1+\alpha/2}}^\infty \exp\left(-c_2t^{(2-\alpha)/(2+\alpha)}\right)\,dt\\
&=:J_1+J_2+J_3.
\end{align*}

By the change of variable $s=\frac{t}{|x-y|^2}$,
\begin{align*}
J_1&=c_1|x-y|^{-(d-2)}\int_0^{\frac{(1+|y|)^{\alpha/2}}{|x-y|}}s^{-d/2}e^{-{c_2}/{s}}\,ds\\
&\le c_3|x-y|^{-(d-2)}\exp\left(-\frac{c_4|x-y|}{(1+|y|)^{\alpha/2}}\right)\cdot
\begin{cases}
1,\ \ & d\ge 3,\\
1+\left(\log\left(\frac{(1+|y|)^{\alpha/2}}{|x-y|}\right)\right)_+,\ &d=2.
\end{cases}
\end{align*}
Applying the change of variable $s=\frac{t}{(1+|y|)^\alpha}$, we obtain that for $d\ge 3$,
\begin{align*}
J_2&=c_1(1+|y|)^{-\frac{\alpha(d-2)}{2}}\int_{\frac{|x-y|}{(1+|y|)^{\alpha/2}}}^{(1+|y|)^{1-\alpha/2}}s^{-d/2}e^{-c_2 s}\,ds\\
&\le c_5 (1+|y|)^{-\frac{\alpha(d-2)}{2}}\exp\left(-\frac{c_2|x-y|}{(1+|y|)^{\alpha/2}}\right)\int_{\frac{|x-y|}{(1+|y|)^{\alpha/2}}}^{(1+|y|)^{1-\alpha/2}}s^{-d/2}ds\\
&\le c_6(|x-y|(1+|y|)^{\alpha/2})^{-d/2+1}\exp\left(-\frac{c_2|x-y|}{(1+|y|)^{\alpha/2}}\right)\\
&=c_6|x-y|^{-(d-2)}\left(\frac{|x-y|}{(1+|y|)^{\alpha/2}}\right)^{d/2-1}\exp\left(-\frac{c_2|x-y|}{(1+|y|)^{\alpha/2}}\right)\\
&\le c_7|x-y|^{-(d-2)}\exp\left(-\frac{c_2|x-y|}{(1+|y|)^{\alpha/2}}\right);
\end{align*}
while for $d=2$,
\begin{align*}
J_2&=c_1\int_{\frac{|x-y|}{(1+|y|)^{\alpha/2}}}^{(1+|y|)^{1-\alpha/2}}s^{-1}e^{-c_2 s}\,ds\\
&=c_1\int_{\frac{|x-y|}{(1+|y|)^{\alpha/2}}}^{\max\left\{\frac{|x-y|}{(1+|y|)^{\alpha/2}},1\right\}}s^{-1}e^{-c_2 s}\,ds+c_1
\int_{{\max\left\{\frac{|x-y|}{(1+|y|)^{\alpha/2}},1\right\}}}^{(1+|y|)^{1-\alpha/2}}s^{-1}e^{-c_2 s}\,ds\\
&\le c_{8}\left(1+\left(\log\left(\frac{(1+|y|)^{\alpha/2}}{|x-y|}\right)\right)_+\right)
\exp\left(-\frac{c_{9}|x-y|}{(1+|y|)^{\alpha/2}}\right).
\end{align*}
Furthermore, we have
\begin{align*}
J_3&\le c_{10} \exp\left(-c_{11}\left((1+|y|)^{1+\alpha/2}\right)^{(2-\alpha)/(2+\alpha)}\right)=
c_{10}\exp\left(-c_{11}(1+|y|)^{1-\alpha/2}\right)\\
&\le c_{12}(1+|y|)^{-(d-2)}\exp\left(-\frac{c_{11}(1+|y|)^{1-\alpha/2}}{2}\right)\le
c_{13}|x-y|^{-(d-2)}\exp\left(-\frac{c_{11}|x-y|}{2(1+|y|)^{\alpha/2}}\right),
\end{align*}
where in the last inequality we used the fact that $|x-y|\le 1+|y|$.

Combining with all the estimates for $J_1$, $J_2$ and $J_3$, we prove the
desired upper bound when $|x-y|\le 1+|y|$.

Next we turn to the lower bound. According to \eqref{t1-3-3}, it holds that
\begin{align*}
G(x,y)&\ge c_{14}\int_{0}^{|x-y|(1+|y|)^{\alpha/2}}t^{-d/2}\exp\left(-\frac{c_{15}|x-y|^2}{t}\right)dt\\
&=c_{14} |x-y|^{-(d-2)}\int_0^{\frac{(1+|y|)^{\alpha/2}}{|x-y|}}s^{-d/2}e^{-{c_{15}}/{s}}\,ds\\
&\ge c_{16}|x-y|^{-(d-2)}\exp\left(-\frac{c_{17}|x-y|}{(1+|y|)^{\alpha/2}}\right)\cdot
\begin{cases}
1,\ \ & d\ge 3,\\
1+\left(\log\left(\frac{(1+|y|)^{\alpha/2}}{|x-y|}\right)\right)_+,\ &d=2.
\end{cases}
\end{align*}

Therefore, by all the conclusions above, we have proved the desired conclusion.
\end{proof}

\section{The case for negative potentials}
In this section, \emph{we always assume that \eqref{e1-1a-} holds}. We aim to prove Theorem \ref{t1-2}. Without any mention throughout the proof below we suppose that $x,y\in \R^d$ so that $|x|\le |y|$.

\subsection{Upper bound}

\begin{lemma}\label{l4-2}
There exist positive constants $C_{1}$, $C_{2}$ and $C_{3}$ such that for all $t>1$ and $x,y\in \R^d$,
\begin{equation}\label{l4-2-1}
\begin{split}
p(t,x,y)\le
C_{1}t^{- {d}/{2}}\exp\left(\max\left\{\frac{C_2t}{(1+|x|)^\alpha},C_2t-\frac{C_3(1+|x|^2)}{t}\right\}\right).
\end{split}
\end{equation}
\end{lemma}
\begin{proof}
Without loss of generality we can assume that $|x|\ge 4$; otherwise the assertion immediately follows from
\eqref{note-}.
Set $U:=B(x,{|x|}/{2})$. According to \eqref{e1-4} again, for any given $f\in C_c(\R^d)$ and $t>0$,
\begin{align*}
T_t^V f(x)&=\Ee_x\left[\exp\left(-\int_0^t V(B_s)\,ds\right)f(B_t)\I_{\{\tau_U>t\}}\right]+
\Ee_x\left[\exp\left(-\int_0^t V(B_s)\,ds\right)f(B_t)\I_{\{\tau_U\le t\}}\right]\\
&=:I_1+I_2.
\end{align*}

By \eqref{e1-1a-}, $$\sup_{z\in U}(-V(z))\le \frac{c_1}{(1+|x|)^\alpha},$$ which implies that
\begin{align*}
I_1&\le  \exp\left(\frac{c_1t}{(1+|x|)^\alpha}\right)\Ee_x\left[f(B_t)\I_{\{\tau_U>t\}}\right]= \exp\left(\frac{c_1t}{(1+|x|)^\alpha}\right)\int_U q_U(t,x,z)f(z)\,dz\\
&\le c_2t^{-d/2}\exp\left(\frac{c_1t}{(1+|x|)^\alpha}\right)\int_{\R^d}f(z)\,dz.
\end{align*}
Here in the last inequality above we have used the fact that
\begin{align*}
q_U(t,x,z)\le q(t,x,z)\le c_2t^{-d/2},\quad t>0.
\end{align*}

For any $t\ge1$, we have
\begin{align*}
I_2&\le \Ee_x\left[\exp\left(-\int_0^t V(B_s)\,ds\right)f(B_t)\I_{\{\tau_U\le t,B_t\in B(x,{|x|}/{3})\}}\right]\\
&\quad +\Ee_x\left[\exp\left(-\int_0^t V(B_s)\,ds\right)f(B_t)\I_{\{\tau_U\le t,B_t\notin B(x,{|x|}/{3})\}}\right]\\
&\le e^{c_3t}\Big(\Ee_x\left[f(B_t)\I_{\{\tau_U\le t,B_t\in B(x,{|x|}/{3})\}}\right]+\Ee_x\left[f(B_t)\I_{\{\tau_U\le t,B_t\notin B(x,{|x|}/{3})\}}\right]\Big)\\
&\le t^{-d/2}e^{2c_3t}\Big(\Ee_x\left[f(B_t)\I_{\{\tau_U\le t,B_t\in B(x,{|x|}/{3})\}}\right]+\Ee_x\left[f(B_t)\I_{\{\tau_U\le t,B_t\notin B(x,{|x|}/{3})\}}\right]\Big)\\
&=:t^{-d/2}e^{2c_3t}\left(I_{21}+I_{22}\right).
\end{align*}
It holds that
\begin{equation*}
\begin{split}
I_{21}&\le \Ee_x\left[\Ee_{B_{\tau_U}}\left(f(B_{t-\tau_U})\I_{\{B_{t-\tau_U}\in B(x,{|x|}/{3})\}}\right)
\I_{\{\tau_U\le t\}}
\right]\le \sup_{s\in (0,t],u\in \partial U}\int_{B(x,{|x|}/{3})}q(s,u,z)f(z)\,dz\\
&\le c_4\sup_{s\in (0,t]}\left( s^{-d/2}\exp\left(-\frac{c_5|x|^2}{s}\right)\right)\cdot \int_{\R^d}f(z)\,dz\le c_6\exp\left(-\frac{c_7|x|^2}{t}\right)\cdot \int_{\R^d}f(z)\,dz,
\end{split}
\end{equation*}
where in the last inequality we have used the fact that for every $x\in \R^d$ with $|x|\ge 1$
\begin{align*}
\sup_{s\in (0,t]}\left( s^{-d/2}\exp\left(-\frac{c_5|x|^2}{2s}\right)\right)\le
\sup_{s>0}\left( s^{-d/2}\exp\left(-\frac{c_5}{2s}\right)\right)\le c_8.
\end{align*}
On the other hand,
\begin{equation*}
\begin{split}
I_{22}&\le  \Ee_x\left[f(B_t)\I_{\{B_t\notin B(x,{|x|}/{3})\}}\right]=\int_{B(x,{|x|}/{3})^c}q(t,x,z)f(z)\,dz\\
&\le c_9t^{-d/2}\exp\left(-\frac{c_{10}|x|^2}{t}\right)\int_{\R^d}f(z)\,dz.
\end{split}
\end{equation*}

Combining with all the estimates above yields that for every $f\in C_c(\R^d)$ and $t>1$,
\begin{align*}
T_t^Vf(x)&=\int_{\R^d}p(t,x,z)f(z)\,dz\le
c_{11}t^{-d/2}\exp\left(\max\left\{\frac{c_{12}t}{(1+|x|)^\alpha},c_{12}t-\frac{c_{13}(1+|x|^2)}{t}\right\}\right).
\int_{\R^d}f(z)\,dz.
\end{align*}
This proves the desired conclusion \eqref{l4-2-1}.
\end{proof}

\begin{lemma}\label{l4-3}
For any positive constant $C_0$, there exist positive constants $C_4$, $C_5$, $C_6$ and $C_7$ such that for all $t>2$ and $x,y\in \R^d$ with $|x-y|> C_0 t^{1/2}$,
\begin{equation}\label{l4-3-1}
 p(t,x,y)\le C_4t^{- {d}/{2}}\exp\left(-\frac{C_5|x-y|^2}{t}\right)\exp\left(\max\left\{\frac{C_6t}{(1+|x|)^\alpha},C_6t-\frac{C_7(1+|x|^2)}{t}\right\}\right).
\end{equation}
\end{lemma}
\begin{proof}
The proof is split into two cases.

{\bf Case 1: $|x-y|\ge {(1+|x|)}/{8}$.}\,\,Define $U=B(y,{|x-y|}/{3})$. By the semigroup property, it holds that for all $t>0$
\begin{align*}
p(t+1,x,y)&=p(t+1,y,x)=\int_{\R^d}p(t,y,z)p(1,z,x)\,dz=T_t^V\left(p(1,\cdot,x)\right)(y)\\
&=\Ee_y\left[\exp\left(-\int_0^t V(B_s)\,ds\right)p(1,B_t,x)\right]\\
&=\Ee_y\left[\exp\left(-\int_0^t V(B_s)\,ds\right)p(1,B_t,x)\I_{\{\tau_U>t\}}\right]\\
&\quad+
\Ee_y\left[\exp\left(-\int_0^t V(B_s)\,ds\right)p(1,B_t,x)\I_{\{\tau_U\le t\}}\right]\\
&=:I_1+I_2.
\end{align*}

Note that, for all $z\in U$, $|z|\ge |y|-{|x-y|}/{3}\ge |y|-({|x|+|y|})/{3}\ge {|y|}/{3}$, and so
$$\sup_{z\in U}(-V(z))\le \frac{c_1}{(1+|y|)^\alpha}.$$ On the other hand, for all $z\in U$, $|z-x|\ge |x-y|-|y-z|\ge 2|x-y|/3$, and so
\begin{align}\label{l4-3-2a}
p(1,z,x)\le c_2q(1,z,x)\le c_2\exp\left(-c_3|z-x|^2\right)\le c_2
\exp\left(-c_4|x-y|^2\right),\quad z\in U.
\end{align}
Thus,
\begin{align*}
I_1&\le \exp\left(t\sup_{z\in U}(-V(z))\right)\cdot \sup_{z\in U}p(1,z,x)\le c_2\exp\left(\frac{c_1 t}{(1+|y|)^\alpha}\right)\cdot \exp\left(-c_4|x-y|^2\right)\\
&\le  c_5t^{-d/2}\exp\left(-\frac{c_6|x-y|^2}{t}\right)\cdot\exp\left(\frac{c_1 t}{(1+|y|)^\alpha}\right),
\end{align*}
where in the last inequality we used the fact that
\begin{equation}\label{l4-3-2}
\begin{split}
\exp\left(-c_4|x-y|^2\right)&\le c_7|x-y|^{-d}\exp\left(-\frac{c_4|x-y|^2}{2t}\right)\\
&=c_7t^{-d/2}\left(\frac{|x-y|^2}{t}\right)^{-d/2}\exp\left(-\frac{c_4|x-y|^2}{2t}\right)\\
&\le c_8t^{-d/2}\exp\left(-\frac{c_6|x-y|^2}{t}\right),\quad
|x-y|>C_0t^{1/2},\ t>1.
\end{split}
\end{equation}

Meanwhile, for all $t>1$, it holds that
\begin{align*}
I_2&\le \Ee_y\left[\exp\left(-\int_0^t V(B_s)\,ds\right)p(1,B_t,x)\I_{\{\tau_U\le t,B_t\in B(y,{|x-y|}/{3})\}}\right]\\
&\quad +\Ee_y\left[\exp\left(-\int_0^t V(B_s)\,ds\right)p(1,B_t,x)\I_{\{\tau_U\le t,B_t\notin B(y,{|x-y|}/{3})\}}\right]\\
&\le c_{9}e^{c_{10}t}\Big(\Pp_y\left(\tau_U\le t, B_t\in B(y,{|x-y|}/{3})\right)+\Pp_y\left(\tau_U\le t, B_t\notin B(y,{|x-y|}/{3})\right)\Big)\\
&\le c_{9}t^{-d/2}e^{2c_{10}t}\Big(\Pp_y\left(\tau_U\le t, B_t\in B(y,{|x-y|}/{3})\right)+\Pp_y\left(\tau_U\le t, B_t\notin B(y,{|x-y|}/{3})\right)\Big)\\
&=:c_{9}t^{-d/2}e^{2c_{10}t}\left(I_{21}+I_{22}\right).
\end{align*}
Here in the third inequality we have used the fact that
\begin{align*}
\sup_{z \in \R^d}p(1,z,x)\le e^{c_{11}}\sup_{z \in \R^d}q(1,z,x)\le c_{12}.
\end{align*}
According to the arguments of \eqref{l2-2-3} and \eqref{l2-2-4}
(see also the arguments for estimates of $I_{21}$ and $I_{22}$ in the proof of Lemma \ref{l4-2}), we can obtain
\begin{align*}
I_{21}+I_{22}&\le c_{13}\exp\left(-\frac{c_{14}|x-y|^2}{t}\right),\quad
|x-y|>C_0t^{1/2},\ t>1.
\end{align*}

Therefore, combining with all the estimates above yields that for all $t>1$,
\begin{align*}
 p(t+1,x,y)&\le c_{15}t^{- {d}/{2}}\exp\left(-\frac{c_{16}|x-y|^2}{t}\right)\exp\left(\max\left\{\frac{c_{17}t}{(1+|y|)^\alpha},c_{17}t-\frac{c_{18}|x-y|^2}{t}\right\}\right)\\
 &\le c_{15}t^{- {d}/{2}}\exp\left(-\frac{c_{16}|x-y|^2}{t}\right)\exp\left(\max\left\{\frac{c_{17}t}{(1+|x|)^\alpha},c_{17}t-\frac{c_{19}(1+|x|^2)}{t}\right\}\right),
\end{align*}
where the last inequality follows from the fact that $|x|\le |y|$ and $|x-y|\ge (1+|x|)/8$.

{\bf Case 2: $|x-y|\le  {(1+|x|)}/{8}$.}\,\,Define $U_1:=B\left(x,{|x-y|/2} \right)$ and $U_2=B\left(x,({1+|x|})/{2}\right)$. It is obvious that $U_1\subset U_2$. Following the arguments in {\bf Case 1}, we have for all $t>0$,
\begin{align*}
p(t+1,x,y)
&=\int_{\R^d}p(t,x,z)p(1,z,y)dz=T_t^V\left(p(1,\cdot,y)\right)(x)\\
&=\Ee_x\left[\exp\left(-\int_0^t V(B_s)ds\right)p(1,B_t,y)\right]\\
&=\Ee_x\left[\exp\left(-\int_0^t V(B_s)ds\right)p(1,B_t,y)\I_{\{\tau_{U_1}>t\}}\right]\\
&\quad+
\Ee_x\left[\exp\left(-\int_0^t V(B_s)ds\right)p(1,B_t,y)\I_{\{\tau_{U_1}\le t,\tau_{U_2}>t\}}\right]\\
&\quad +\Ee_x\left[\exp\left(-\int_0^t V(B_s)ds\right)p(1,B_t,y)\I_{\{\tau_{U_2}\le t\}}\right]\\
&=:J_1+J_2+J_3.
\end{align*}

For all $z\in U_1$, $|z|\ge |x|-{|x-y|/2}\ge |x|-({1+|x|})/{16}$,  $$\sup_{z\in U_1}(-V(z))\le \frac{c_1}{(1+|x|)^\alpha}.$$
Combining this with the arguments for \eqref{l4-3-2a} and \eqref{l4-3-2} yields that
\begin{align*}
J_1&\le \exp\left(t\sup_{z\in U_1}(-V(z))\right)\cdot \sup_{z\in U_1}p(1,z,y)\le  c_2 t^{-d/2}\exp\left(-\frac{c_3|x-y|^2}{t}\right)\cdot\exp\left(\frac{c_1 t}{(1+|x|)^\alpha}\right).
\end{align*}

Noting that $$\sup_{z\in U_2}(-V(z))\le \frac{c_4}{(1+|x|)^\alpha},$$ we can define
\begin{align*}
J_2&\le \exp\left(\frac{c_4 t}{(1+|x|)^\alpha}\right)\cdot
\Ee_x\left[p(1,B_t,y)\I_{\{\tau_{U_1}\le t\}}\right]\\
&\le \exp\left(\frac{c_4 t}{(1+|x|)^\alpha}\right)\cdot\left(
\Ee_x\left[p(1,B_t,y)\I_{\{\tau_{U_1}\le t,  B_t\in B(x,{|x-y|}/{3})\}}\right]+
\Ee_x\left[p(1,B_t,y)\I_{\{\tau_{U_1}\le t, B_t\notin B(x,{|x-y|}/{3})\}}\right]\right)\\
&=:\exp\left(\frac{c_4 t}{(1+|x|)^\alpha}\right)\cdot\left(J_{21}+J_{22}\right).
\end{align*}
Applying the strong Markov property, we get
\begin{align*}
J_{21}&\le \Ee_x\left[\Ee_{B_{\tau_{U_1}}}\left[p\left(1,B_{t-\tau_{U_1}},y\right)\I_{\{B_{t-\tau_{U_1}}\in B(x,{|x-y|}/{3})\}}\right]\I_{\{\tau_{U_1}\le t\}}\right]\\
&\le \sup_{s\in (0,t],u\in \partial U_1}\int_{B(x,{|x-y|}/{3})}q(s,u,z)p(1,z,y)\,dz\\
&\le c_5\sup_{s\in (0,t]}\left( \left(\frac{|x-y|^2}{s}\right)^{d/2}\exp\left(-\frac{c_6|x-y|^2}{s}\right)\right)\cdot |x-y|^{-d}\int_{\R^d}p(1,z,y)\,dz\\
&\le c_7\exp\left(-\frac{c_8|x-y|^2}{t}\right)\cdot |x-y|^{-d}\le c_9t^{-d/2}\exp\left(-\frac{c_{10}|x-y|^2}{t}\right),
\end{align*}
where in the fourth inequality we have used the fact that for every $x,y\in \R^d$ and $t>0$ with $|x-y|\ge C_0t^{1/2}$,
\begin{align*}
\sup_{s\in (0,t]}\left( \left(\frac{|x-y|^2}{s}\right)^{d/2}\exp\left(-\frac{c_6|x-y|^2}{s}\right)\right)\le
c_{11}\exp\left(-\frac{c_8|x-y|^2}{t}\right).
\end{align*}
On the other hand, we have
\begin{equation*}
\begin{split}
J_{22}&\le \Ee_x\left[p(1,B_t,y)\I_{\{B_t\notin B(x,{|x-y|}/{3})\}}\right]=\int_{B(x,{|x-y|}/{3})^c}q(t,x,z)p(1,z,y)\,dz\\
&\le c_{12}\int_{\R^d}q(t,x,z)q(1,z,y)\,dz=c_{12}q(t+1,x,y)\le c_{13}t^{-d/2}\exp\left(-\frac{c_{14}|x-y|^2}{t}\right).
\end{split}
\end{equation*}

Similarly,  we define
\begin{align*}
J_3&\le e^{c_{15}t}\Ee_x\left[p(1,B_t,y)\I_{\{\tau_{U_2}\le t\}}\right]\\
&=e^{c_{15}t}\Big(\Ee_x\left[p(1,B_t,y)\I_{\{\tau_{U_2}\le t,B_t\in B(x,({1+|x|})/{3})\}}\right]+\Ee_x\left[p(1,B_t,y)\I_{\{\tau_{U_2}\le t, B_t\notin B(x,({1+|x|})/{3})\}}\right]\Big)\\
&=:e^{c_{15}t}\left(J_{31}+J_{32}\right).
\end{align*}
Using the strong Markov inequality again, we deduce
\begin{align*}
J_{31}&\le \Ee_x\left[\Ee_{B_{\tau_{U_2}}}\left[p (1,B_{t-\tau_{U_2}},y )\I_{\{B_{t-\tau_{U_2}}\in B(x,({1+|x|})/{3})\}}\right]\I_{\{\tau_{U_2}\le t\}}\right]\\
&\le \sup_{s\in (0,t],u\in \partial U_2}\int_{B(x,({1+|x|})/{3})}q(s,u,z)p(1,z,y)\,dz\\
&\le c_{16}\sup_{s\in (0,t]}\left( \left(\frac{(1+|x|)^2}{s}\right)^{d/2}\exp\left(-\frac{c_{17}(1+|x|)^2}{s}\right)\right)\cdot (1+|x|)^{-d}\int_{\R^d}p(1,z,y)\,dz\\
&\le c_{18}\exp\left(-\frac{c_{19}(1+|x|)^2}{t}\right)\cdot (1+|x|)^{-d}\\
&\le c_{20}t^{-d/2}\exp\left(-\frac{c_{21}|x-y|^2}{t}\right)\exp\left(-\frac{c_{21}(1+|x|)^2}{t}\right),
\end{align*}
where the last inequality follows from the fact that $C_0t^{1/2}\le |x-y|\le {(1+|x|)}/{8}$. On the other hand, following the arguments of $J_{22}$, we obtain
\begin{align*}
J_{32}&\le \Ee_x\left[p(1,B_t,y)\I_{\{B_t\notin B(x,({1+|x|})/{3})\}}\right]=\int_{B(x,({1+|x|})/{3})^c}q(t,x,z)p(1,z,y)\,dz\\
&\le c_{22}\sup_{z\notin B(x,({1+|x|})/{3})}q(t,x,z)\cdot\int_{B(x,({1+|x|})/{3})^c}q(1,z,y)\,dz\\
&\le c_{23}t^{-d/2}\exp\left(-\frac{c_{24}(1+|x|)^2}{t}\right)\\
&\le c_{25}t^{-d/2}\exp\left(-\frac{c_{26}|x-y|^2}{t}\right)\exp\left(-\frac{c_{26}(1+|x|)^2}{t}\right),
\end{align*}
where in the last inequality we have used the fact that $C_0t^{1/2}\le |x-y|\le  {(1+|x|)}/{8}$ again.

Now, combining with all the estimates of $J_1,J_2,J_3$ above yields the desired conclusion \eqref{l4-3-1} for
the case that $|x-y|\ge C_0t^{1/2}$.
\end{proof}

Summarising Lemmas \ref{l4-2} and \ref{l4-3}, we get
\begin{proposition}\label{p4-1}
There exist positive constants $C_8$, $C_9$, $C_{10}$ and $C_{11}$ such that for all $t>2$ and $x,y\in \R^d$,
\begin{equation}\label{p4-1-1}
 p(t,x,y)\le C_{8}t^{- {d}/{2}}\exp\left(-\frac{C_9|x-y|^2}{t}\right)\exp\left(\max\left\{\frac{C_{10}t}{(1+|x|)^\alpha},C_{10}t-\frac{C_{11}(1+|x|^2)}{t}\right\}\right).
\end{equation}
\end{proposition}

Finally, we are concerned on the case that $\alpha\in (2,+\infty)$ and $K_2$ is small.

\begin{lemma}\label{l4-6}
For every $\alpha\in (2,+\infty)$, it holds that
\begin{equation}
\label{e:equ1}\begin{split} \int_{\R^d} \frac{s^{-d/2}}{(1+|z|)^\alpha}\exp\left(-\frac{|x-z|^2}{s}\right)\,dz
\le C_{12}\begin{cases}
\frac{1+\log (2+|x|)\I_{\{d=\alpha\}}}
{(1+|x|)^{\min\{d,\alpha\}}},\qquad \qquad \quad\,\,\, & s<(1+|x|)^2,\\
s^{-\min\{d,\alpha\}/2}\left(1+\log (1+s)\I_{\{d=\alpha\}}\right),\quad & s\ge (1+|x|)^2.
\end{cases}
\end{split}
\end{equation}
\end{lemma}
\begin{proof}
The proof is split into two cases.

{\bf Case 1: $s\le (1+|x|)^2$.}\,\,
By the change of variable with $\tilde z=(x-z)/\sqrt{s}$,
it holds that the left hand side of \eqref{e:equ1} equals to
\begin{align*}\int_{\R^d}\frac{1}{(1+|x-\sqrt{s}z|)^\alpha}\exp\left(-|z|^2\right)\,dz=&\int_{\{|x-\sqrt{s}z|\le (1+|x|)/2\}}\frac{1}{(1+|x-\sqrt{s}z|)^\alpha}\exp\left(-  |z|^2\right)\,dz\\
&+\int_{\{|x-\sqrt{s}z|\ge (1+|x|)/2\}}\frac{1}{(1+|x-\sqrt{s}z|)^\alpha}\exp\left(-  |z|^2\right)\,dz\\
=&:I_1+I_2.\end{align*}

It is clear that $$I_2\le \frac{c_1}{(1+|x|)^\alpha}\int_{\R^d} e^{- |z|^2}\,dz\le  \frac{c_2}{(1+|x|)^\alpha}.$$
For $I_1$, we apply the following fact $$\{z\in \R^d: |x-\sqrt{s}z|< (1+|x|)/2\}\subset \cup_{k=0}^{[(1+|x|)/2]}\{z\in \R^d: k\le |x-\sqrt{s}z|< k+1\}=:\cup_{k=0}^{[(1+|x|)/2]} A_k.$$ In particular,
$$|A_k|\le (c_3/\sqrt{s})((k+1)/\sqrt{s})^{d-1}=c_3(k+1)^{d-1}s^{-d/2}.$$
Moreover, for every $z\in A_k$, $|x|\ge 2$ and $0\le k\le [(1+|x|)/2]$,
$$|z|\ge \frac{|x|}{\sqrt{s}}-\frac{k+1}{\sqrt{s}}\ge \frac{c_4|x|}{\sqrt{s}},$$ which implies $e^{-|z|^2}\le e^{-c_5(1+|x|)^2/s}$. Thus, for all $|x|\ge 2$ and $0\le k\le [(1+|x|)/2]$,
\begin{align*}\int_{A_k}\frac{1}{(1+|x-\sqrt{s}z|)^\alpha}\exp\left(- |z|^2\right)\,dz\le &\frac{c_6}{(1+k)^\alpha}\frac{(k+1)^{d-1}}{s^{d/2}}e^{-c_5(1+|x|)^2/s},\end{align*} and so
\begin{equation}\label{e:eproof1}\begin{split} I_1\le &\sum_{k=0}^{[(1+|x|)/2]}\int_{A_k}\frac{1}{(1+|x-\sqrt{s}z|)^\alpha}\exp\left(-|z|^2\right)\,dz\\
\le&\sum_{k=0}^{[(1+|x|)/2]}\frac{c_6}{(1+k)^\alpha}\frac{(k+1)^{d-1}}{s^{d/2}}e^{-c_5(1+|x|)^2/s}\\
\le& c_7(1+|x|)^{-\min\{d,\alpha\}}
\left(1+\log (2+|x|)\I_{\{d=\alpha\}}\right)
\left(\frac{1+|x|}{\sqrt{s}}\right)^d e^{-c_5(1+|x|)^2/s}\\
 \le& c_8(1+|x|)^{-\min\{d,\alpha\}}\left(1+\log (2+|x|)\I_{\{d=\alpha\}}\right).\end{split} \end{equation}
Here the last inequality follows from the property that
\begin{align*}
\left(\frac{1+|x|}{\sqrt{s}}\right)^d e^{-c_5(1+|x|)^2/s}\le \sup_{a\ge 1}a^d e^{-c_5 a}\le c_9,
\end{align*}
where we have also used the fact $s\le (1+|x|)^2$.
Furthermore, it is easy to see that when $|x|\le2$, $I_1$ is bounded, and so \eqref{e:eproof1} still holds. Thus,
\eqref{e:eproof1} is satisfied for all $x\in \R^d$.

Combining with both estimates above for $I_1$ and $I_2$ yields the desired assertion.

{\bf Case 2: $s\ge (1+|x|)^2$.}\,\,
We still define $I_1$ and $I_2$ by the same way as that in the proof of {\bf Case 1}.
First, according to the third inequality in \eqref{e:eproof1}, we find that
\begin{align*}
I_1&\le c_1(1+|x|)^{-\min\{d,\alpha\}}\left(1+\log (1+|x|)\I_{\{d=\alpha\}}\right)\left(\frac{1+|x|}{\sqrt{s}}\right)^d e^{-c_2(1+|x|)^2/s}\\
&\le c_3 (1+|x|)^{ \max\{d-\alpha,0\}}\left(1+\log (1+|x|)\I_{\{d=\alpha\}}\right)s^{-d/2}\\
&\le c_3 s^{-\min\{d,\alpha\}/2}\left(1+\log (1+|x|)\I_{\{d=\alpha\}}\right)\left((1+|x|)/\sqrt{s}\right)^{\max\{d-\alpha,0\}}\\
&\le c_3 s^{-\min\{d,\alpha\}/2}\left(1+\log (1+|x|)\I_{\{d=\alpha\}}\right).
\end{align*}
Here the second and fourth inequality follow from the fact $s\ge (1+|x|)^2$.

Now, we turn to the estimate for $I_2$. For this, we write
\begin{align*}\{z\in \R^d:|x-\sqrt{s}z|\ge (1+|x|)/2\}\subset & \{z\in \R^d:(1+|x|)/2\le |x-\sqrt{s}z|\le 2(1+|x|)\}\\
& \cup \{z\in \R^d:  |x-\sqrt{s}z|\ge 2(1+|x|)\}\\
\subset & B_0 \cup ( \cup_{k=[2(1+|x|)]}^\infty B_k),\end{align*} where $B_0=\{z:(1+|x|)/2\le |x-\sqrt{s}z|\le 2(1+|x|)\}$ and $B_k=\{z: k\le |x-\sqrt{s}z|< k+1\}$ for $k\ge [2(1+|x|)]$. Set
$$I_{20}:= \int_{B_0}\frac{1}{(1+|x-\sqrt{s}z|)^\alpha}\exp\left(-|z|^2\right)\,dz,\quad I_{2k}=\int_{B_k}\frac{1}{(1+|x-\sqrt{s}z|)^\alpha}\exp\left(-|z|^2\right)\,dz.$$
Since $|B_0|\le c_4(1+|x|)^d/s^{d/2}$, $$I_{20}\le c_5 (1+|x|)^{d-\alpha}/s^{d/2}\le c_6 s^{-\min\{d,\alpha\}/2}\left((1+|x|)/\sqrt{s}\right)^{\max\{d-\alpha,0\}}
\le c_6s^{-\min\{d,\alpha\}/2}.$$ On the other hand,
for any $k\ge [2(1+|x|)]$ and $z\in B_k$ it holds that $|z|\ge k/\sqrt{s}-|x|/ \sqrt{s} \ge c_7 k/\sqrt{s}$. Thus,
$$I_{2k}\le c_8 (1+k)^{-\alpha}e^{-c_9 k^2/s}|B_k|\le c_{10}(1+k)^{-\alpha} (k/\sqrt{s})^{d-1}(1/\sqrt{s}) e^{-c_{11} k^2/s},$$ and so
\begin{align*} \sum_{k=[2(1+|x|)]}^\infty I_{2k}&\le c_{10} \sum_{k=[2(1+|x|)]}^\infty  (1+k)^{-\alpha} (k/\sqrt{s})^{d-1}(1/\sqrt{s} )e^{-c_{11}k^2/s} \\
&=c_{10}\left(\sum_{k=[2(1+|x|)]}^{2[\sqrt{s/c_{11}}]} \,\, +\,\, \sum_{2[\sqrt{s/c_{11}}]+1}^\infty\right) (1+k)^{-\alpha} (k/\sqrt{s})^{d-1}(1/\sqrt{s}) e^{-c_{11}  k^2/s}\\
&\le c_{12} \left(s^{-\min\{d,\alpha\}/2}\left(1+\log (1+s)\I_{\{d=\alpha\}}\right)+s^{-\alpha/2}\right)\\
&\le c_{13}s^{-\min\{d,\alpha\}/2}\left(1+\log (1+s)\I_{\{d=\alpha\}}\right).
\end{align*}
Therefore, putting all the estimates above together, we arrive at the desired assertion again.
\end{proof}

\begin{lemma}\label{lemma4.2}
Suppose that $\alpha\in (2,+\infty)$ and $d\ge 3$. For any $0<a<b$, there is a constant $C_{13}>0$ so that for all $x,y\in \R^d$,
\begin{equation}\label{l4-7-1}
\begin{split}
&\int_0^t\int_{\R^d} \frac{1}{(t-s)^{d/2}}\exp\left(-\frac{b|x-z|^2}{t-s}\right)\frac{1}{(1+|z|)^\alpha}\frac{1}{s^{d/2}}\exp\left(-\frac{a|z-y|^2}{s}\right)\,dz\,ds\\
&\le C_{13}t^{-d/2}\exp\left(-\frac{a|x-y|^2}{t}\right).
\end{split}
\end{equation}
\end{lemma}
\begin{proof}
The proof is partly inspired by that of \cite[Lemma 3.1]{Z3}. We write
\begin{align*}&\int_0^t\int_{\R^d} \frac{1}{(t-s)^{d/2}}\exp\left(-\frac{b|x-z|^2}{t-s}\right)\frac{1}{(1+|z|)^\alpha}\frac{1}{s^{d/2}}\exp\left(-\frac{a|z-y|^2}{s}\right)\,dz\,ds\\
&=\int_0^{(1-\theta)t}\int_{\R^d} \,\,+\,\,\int_{(1-\theta)t}^t\int_{\R^d}\cdots \,dz\,ds=:J_1+J_2,\end{align*} where $\theta\in (0,1)$ is a constant to be determined later.
By \cite[(3.4) in Lemma 3.1]{Z3}, $$\frac{|x-z|^2}{t-s}+\frac{|z-y|^2}{s}\ge \frac{|x-y|^2}{t},\quad x,y,z\in \R^d,0<s\le t.$$ Then, according to \eqref{e:equ1},
\begin{align*}
J_2=&\int_{(1-\theta)t}^t\int_{\R^d} \frac{1}{(t-s)^{d/2} }\exp\left(-\frac{(b-a)|x-z|^2}{t-s}\right)\frac{1}{(1+|z|)^\alpha}\frac{1}{s^{d/2}}\exp\left(-a\left(\frac{|z-y|^2}{s}+\frac{|x-z|^2}{t-s}\right)\right) \,dz\,ds\\
\le &\frac{1}{((1-\theta)t)^{d/2}}\exp\left(-\frac{a|x-y|^2}{t}\right)\int_{(1-\theta)t}^t\int_{\R^d} \frac{1}{(t-s)^{d/2}}\exp\left(-(b-a)\frac{|x-z|^2}{t-s}\right)\frac{1}{(1+|z|)^\alpha}\,dz\,ds\\
\le & \frac{c_1}{((1-\theta)t)^{d/2}}\exp\left(-\frac{a|x-y|^2}{t}\right)\\
    &  \times\int_{(1-\theta)t}^t \Bigg(\frac{\log (2+|x|)}{(1+|x|)^{\min\{d,\alpha\}}}\I_{\{(1+|x|)^2\ge t-s\}} +(t-s)^{-\min\{d,\alpha\}/2}\log (1+t-s)\I_{\{(1+|x|)^2\le t-s\}}\Bigg)\,ds\\
\le & c_2 t^{-d/2}\exp\left(-a\frac{|x-y|^2}{t}\right),
\end{align*}
where in the last inequality  we have used fact that $\min\{d,\alpha\}>2$.

On the other hand,  it holds that
\begin{align*}
J_1= &\int_0^{(1-\theta)t}\int_{\{|z-x|\ge |x-y|(a/b)^{1/2}\}} \frac{1}{(t-s)^{d/2}}\exp\left(-\frac{b|x-z|^2}{t-s}\right)\frac{1}{(1+|z|)^\alpha}\frac{1}{s^{d/2}}\exp\left(-\frac{a|z-y|^2}{s}\right)\,dz\,ds\\
&+ \int_0^{(1-\theta)t}\int_{\{|z-x|\le |x-y|(a/b)^{1/2}\}} \frac{1}{(t-s)^{d/2}}\exp\left(-\frac{b|x-z|^2}{t-s}\right)\frac{1}{(1+|z|)^\alpha}\frac{1}{s^{d/2}}\exp\left(-\frac{a|z-y|^2}{s}\right)\,dz\,ds\\
=&:J_{11}+J_{12}.\end{align*}
If $|z-x|\ge |x-y|(a/b)^{1/2}$ and $s\in (0,(1-\theta)t)$, then
\begin{align*}
\frac{1}{(t-s)^{d/2}}\exp\left(-\frac{b|x-z|^2}{t-s}\right)&\le \frac{c_3}{t^{d/2}}\exp\left(-\frac{a|x-y|^2}{t}\right).
\end{align*}
Thus,
\begin{align*} J_{11}\le & c_5t^{-d/2}\exp\left(-\frac{a|x-y|^2}{t}\right)\int_0^{(1-\theta)t}\int_{\R^d} s^{-d/2}\exp\left(-\frac{a|z-y|^2}{s}\right)\frac{1}{(1+|z|)^\alpha}\,dz\,ds\\
\le & c_6t^{-d/2}\exp\left(-\frac{a|x-y|^2}{t}\right)\int_0^{(1-\theta)t}
\bigg(\frac{\log (2+|y|)}{(1+|y|)^{\min\{d,\alpha\}}}\I_{\{s\le (1+|y|)^2\}}\\
&\qquad\qquad\qquad \qquad\qquad\qquad \qquad\qquad+
s^{-\min\{d,\alpha\}/2}
\log(1+s)\I_{\{s>(1+|y|)^2\}}\bigg)\,ds\\
\le& c_7t^{-d/2}\exp\left(-\frac{a|x-y|^2}{t}\right).
\end{align*}

 If $|z-x|\le |x-y|(a/b)^{1/2}$, then
 $|z-y|\ge |x-y|-|z-x|\ge |x-y|\left(1-(a/b)^{1/2}\right).$ Hence, for all $0<s\le (1-\theta)t$,
\begin{align*}\exp\left(-\frac{a|z-y|^2}{s}\right)\le \exp\left(-\frac{a|z-y|^2}{2s}\right)\exp\left(-a(1-(a/b)^{1/2})^2\frac{|x-y|^2}{2(1-\theta)t}\right).
\end{align*}
Now, we choose $\theta\in (0,1)$ such that
$$\frac{(1-(a/b)^{1/2})^2}{2(1-\theta)}=1.$$ So it holds that for any $|z-x|\le |x-y|(a/b)^{1/2}$ and $0<s\le (1-\theta)t$,
$$\exp\left(-\frac{a|z-y|^2}{s}\right)\le  \exp\left(-\frac{a|z-y|^2}{2s}\right)\exp\left(-\frac{a|x-y|^2}{t}\right).$$ This along with \eqref{e:equ1} yields that
\begin{align*}
J_{12}\le & c_{12} (\theta t)^{-d/2}\exp\left(-\frac{a|x-y|^2}{t}\right)\int_0^{(1-\theta)t}\int_{\R^d}\exp\left(-\frac{a|z-y|^2}{2s}\right)s^{-d/2}\frac{1}{(1+|z|)^\alpha}\,dz\,ds\\
\le & c_{13} t^{-d/2}\exp\left(-\frac{a|x-y|^2}{t}\right)
\int_0^{(1-\theta)t}\bigg(\frac{\log (2+|y|)}{(1+|y|)^{\min\{d,\alpha\}}}\I_{\{s\le (1+|y|)^2\}}\\
&\qquad\qquad\qquad\qquad\qquad\qquad\qquad\qquad+s^{-\min\{d,\alpha\}/2}\log(1+s)\I_{\{s>(1+|y|)^2\}}\bigg)\,ds\\
\le & c_{14} t^{-d/2}\exp\left(-\frac{a|x-y|^2}{t}\right).
\end{align*}

Combining all the above estimates together yields the desired conclusion.
\end{proof}

\begin{proposition}\label{p4-3} Suppose that \eqref{e1-1a-} holds with some $\alpha\in (2,+\infty)$ and $K_1,K_2>0$. Then, there exists a constant $K_0>0$ such that if $K_2\le K_0$,
then for any $x,y\in \R^d$ and $t>0$,
\begin{equation}\label{p4-3-1}
p(t,x,y)\asymp q(t,x,y).
\end{equation}
\end{proposition}
\begin{proof}
Since $V(z)\le 0$, we have $p(t,x,y)\ge q(t,x,y)$, so it suffices to prove the upper bound in \eqref{p4-3-1}.

According to the classical Duhamel formula (see \cite{BDS}), we have the following expression for $p(t,x,y)$:
$$p(t,x,y)=q(t,x,y)+\int_0^t\int_{\R^d} q(t-s,x,z)(-V(z))p(s,z,y)\,dz\,ds.$$ Set
$p_0(t,x,y)=q(t,x,y),$ and \begin{equation}\label{e:proof0}p_n(t,x,y)=\int_0^t\int_{\R^d} q(t-s,x,z)(-V(z))p_{n-1}(s,z,y)\,dz\,ds,\quad n\ge 1.\end{equation} Then, for any $t>0$ and $x,y\in \R^d$,
\begin{equation}\label{e:proof01}p(t,x,y)=\sum_{n=0}^\infty p_n(t,x,y).\end{equation}

Suppose that \eqref{e1-1a-} holds. According to \eqref{l4-7-1},  we obtain immediately that (by taking $a={1}/{4}$ and $b={1}/{2}$)
\begin{align*}
p_1(t,x,y)\le c_1K_2t^{-d/2}\exp\left(-\frac{|x-y|^2}{4t}\right),\quad t>0,\ x,y\in \R^d.
\end{align*}

Next, assume that
\begin{equation}\label{p4-3-2}
p_n(t,x,y)\le (c_1K_2)^n t^{-d/2}\exp\left(-\frac{|x-y|^2}{4t}\right),\quad t>0,\ x,y\in \R^d.
\end{equation}
Then, putting \eqref{p4-3-2}, \eqref{e1-1a-}, \eqref{e:proof0} and \eqref{l4-7-1} together, we derive
\begin{align*}
p_{n+1}(t,x,y)&\le (c_1K_2)^n K_2\int_0^t \int_{\R^d}
\frac{1}{(t-s)^{d/2}}\exp\left(-\frac{|x-z|^2}{2(t-s)}\right)\frac{1}{(1+|z|)^\alpha}\frac{1}{s^{d/2}}\exp\left(-\frac{|z-y|^2}{4s}\right)\,dz\,ds\\
&\le (c_1K_2)^{n+1}t^{-d/2}\exp\left(-\frac{|x-y|^2}{4t}\right),\quad t>0,\ x,y\in \R^d.
\end{align*}
Hence, by the induction procedure, we know that \eqref{p4-3-2} holds  for all $n\ge 1$.

Therefore, when $K_2<K_0:=c_1^{-1}$, it holds that
\begin{align*}
p(t,x,y)\le t^{-d/2}\exp\left(-\frac{|x-y|^2}{4t}\right)\cdot \left(\sum_{n=0}^\infty (c_1K_2)^n\right)\le \frac{1}{1-c_1K_2}t^{-d/2}\exp\left(-\frac{|x-y|^2}{4t}\right).
\end{align*}
The proof is complete.
\end{proof}

\subsection{Lower bound}
\begin{lemma}\label{l4-4}
There exists a positive constant $K_0$ so that if \eqref{e1-1a-} holds with $K_1\ge K_0$, then for
any $C_0>0$ there exist
positive constants $C_{1}$, $C_{2}$ and $C_3$ such that for all $t>2$ and $x,y\in \R^d$ with $|x-y|\le C_0t^{1/2}$,
\begin{equation}\label{l4-4-1}
 p(t,x,y)\ge C_{1}t^{- {d}/{2}}\exp\left(\max\left\{\frac{C_2 t}{(1+|x|)^\alpha}, C_2t-\frac{C_3(1+|x|^2)}{t}\right\}\right).
\end{equation}
Moreover, if \eqref{e1-1a-} is satisfied with $\alpha\in (0,2)$ and $K_1,K_2\in (0,+\infty)$ $($without any restriction on the lower bound of $K_1)$, then
\eqref{l4-4-1} still holds for $p(t,x,y)$.
\end{lemma}
\begin{proof} The proof is split into two cases.

{\bf Case 1: $\alpha\in (0,+\infty)$.}\,\,
When $|x|\le |y|\le 2$,  set $U:=B(0,4)$.
For any $f\in B_b(\R^d)$ with {\rm supp}$[f]\subset B(0,2)$ and $t\ge 1/3$, it holds that
 \begin{equation}\label{l4-4-1a}
 \begin{split}
 T_t^Vf(x)=&\Ee_x\left[f(B_t)\exp\left(- \int_0^tV(B_s)\,ds\right)\right]\ge \Ee_x\left[f(B_t)\exp\left( -\int_0^tV(B_s)\,ds\right)\I_{\{t<\tau_U\}}\right] \\
 \ge&\exp\left(t\inf_{z\in U}(-V(z))\right)\Ee_x\left[f(B_t)\I_{\{t<\tau_U\}}\right] \ge
 \exp\left(\frac{K_1t}{5^\alpha}\right)
 \int_{B(0,2)}q_U(t,x,z)f(z)\,dz\\
 \ge& c_1t^{-d/2}\exp\left(\frac{K_1t}{5^\alpha}-c_2t\right)\int_{B(0,2)}f(z)\,dz,
 \end{split}
 \end{equation}
 where in the third inequality we have used the fact
 $\inf_{z\in B(0,4)}(-V(z))\ge {K_1}{5^{-\alpha}}
 $
 that can be seen from \eqref{e1-1a-},
 and the last inequality follows from \eqref{l2-1-2}.
Hence, when $K_1\ge K_0:=2c_25^\alpha$,
for  every  $f\in B_b(\R^d)$ with {\rm supp}$[f]\subset B(0,2)$ and $t\ge 1/3$ we have
\begin{align*}
 T_t^Vf(x)&\ge c_3t^{-d/2}e^{c_2t}\int_{B(0,2)} f(z)\,dz.
 \end{align*}
 This implies that
 \begin{align}\label{l4-4-2}
p(t,x,y)\ge c_3t^{-d/2}e^{c_2t},\quad x,y\in B(0,2),\ t\ge 1/3.
 \end{align}

Now assume that $|x|\ge 2$ and $|x-y|\le C_0t^{1/2}$. Since $V\le 0$, for all $u\in B(x,1)$ and $z\in B(0,2)$,
\begin{equation}\label{l4-4-3}
p\left(\frac{t}{3},u,z\right)\ge q\left(\frac{t}{3},u,z\right)\ge c_4t^{-d/2}\exp\left(-\frac{c_5|u-z|^2}{t}\right)\ge c_4t^{-d/2}\exp\left(-\frac{c_6(1+|x|^2)}{t}\right).
\end{equation}
According to \eqref{l4-4-2} and \eqref{l4-4-3}, we find that for every $u\in B(x,1)$ and $t>1$,
\begin{align*}
p(t,x,u)&\ge \int_{B(0,2)}\int_{B(0,2)}p\left(\frac{t}{3},x,z_1\right)p\left(\frac{t}{3},z_1,z_2\right)p\left(\frac{t}{3},z_2,u\right)
\,dz_1\,dz_2\\
&\ge c_7t^{-3d/2}\exp\left(c_8t-\frac{c_9(1+|x|^2)}{t}\right)\ge c_{10}\exp\left(c_{11}t-\frac{c_9(1+|x|^2)}{t}\right).
\end{align*}
Hence, we know immediately that for every $t>2$
\begin{equation}\label{l4-4-4}
\begin{split}
p(t,x,y)&\ge \int_{B(x,1)}p\left(\frac{t}{2},x,z\right)
p\left(\frac{t}{2},z,y\right)\,dz\ge \int_{B(x,1)}p\left(\frac{t}{2},x,z\right)
q\left(\frac{t}{2},z,y\right)\,dz\\
&\ge c_{12}t^{-d/2}\exp\left(c_{13}t-\frac{c_{14}(1+|x|^2)}{t}\right),
\end{split}
\end{equation}
where the last inequality is due to the fact
\begin{align*}
q\left(\frac{t}{2},z,y\right)\ge c_{15}t^{-d/2},\quad
z\in B(x,1),\ |y-x|\le C_0t^{1/2}.
\end{align*}

Furthermore, if it additionally holds that $t\le c_{16} (1+|x|^2)$, then we set $U:=B(x,2C_0t^{1/2})$. Note that in this case $$\inf_{z\in U}(-V(z))\ge \frac{c_{17}}{(1+|x|)^\alpha}.$$ Then, for any $f\in B_b(\R^d)$ with {\rm supp}$[f]\subset B(x,C_0t^{1/2})$,
 \begin{align*}
 T_t^Vf(x)=&\Ee_x\left[f(B_t)\exp\left( -\int_0^tV(B_s)\,ds\right)\right] \ge \Ee_x\left[f(B_t)\exp\left(- \int_0^tV(B_s)\,ds\right)\I_{\{t<\tau_U\}}\right] \\
 \ge&\exp\left(t\inf_{z\in U}(-V(z))\right)\Ee_x\left[f(B_t)\I_{\{t<\tau_U\}}\right]
 \ge\exp\left(\frac{c_{17}t}{(1+|x|)^\alpha}\right)\int_{B(x,C_0t^{1/2})}q_U(t,x,z)f(z)\,dz\\
 \ge& c_{18}t^{-d/2}\exp\left(\frac{c_{17}t}{(1+|x|)^\alpha}\right)\int_{B(x,C_0t^{1/2})}f(z)\,dz,\end{align*} where in the last inequality we used \eqref{e:Diri}.
Hence,  for every
$t>2$
and $x,y\in \R^d$ with $|x-y|\le C_0t^{1/2}$ and $t\le c_{16} (1+|x|^2)$, it holds that
\begin{align*}
p(t,x,y)\ge c_{18}t^{-d/2}\exp\left(\frac{c_{17}t}{(1+|x|)^\alpha}\right).
\end{align*}

Combining this with \eqref{l4-4-4} yields the desired conclusion for $\alpha\in (0,+\infty)$.

{\bf Case 2: $\alpha\in (0,2)$.}\,\, According to
\eqref{l2-1-2},
$$
q_{B(0,R)}(t,z_1,z_1)\ge c_{1}t^{-d/2}\exp\left(-\frac{c_{2}t}{R^2}\right),\quad t>1,\ R>4,\ z_1,z_2\in B(0,2).
$$
On the other hand, by \eqref{e1-1a-},
\begin{equation*}
\inf_{z\in B(0,R)}(-V(z))\ge \frac{K_1}{(1+R)^\alpha},\quad R>4.
\end{equation*}
Hence, taking $U=B(0,R)$ and applying the same arguments as those for \eqref{l4-4-1a}, we  obtain
\begin{equation}\label{l4-4-5}
\begin{split}
p(t,x,y)&\ge \exp\left(t\cdot \inf_{z\in B(0,R)}(-V(z))\right)\cdot \inf_{z_1,z_2\in B(0,2)}q_{B(0,R)}(t,z_1,z_2)\\
& \ge c_{3}t^{-d/2}\exp\left(\frac{K_1t}{(1+R)^\alpha}-\frac{c_{4}t}{R^2}\right),\quad x,y\in B(0,2),\ R>4,\ t>1.
\end{split}
\end{equation}

Since $\alpha\in (0,2)$, we can find $R_0>4$ large enough so that
\begin{align*}
\frac{K_1}{(1+R_0)^\alpha}>\frac{2c_{4}}{R_0^2}.
\end{align*}
This along with \eqref{l4-4-5} yields that
\begin{align*}
p(t,x,y)\ge c_{5}t^{-d/2}\exp\left(\frac{K_1t}{2(1+R_0)^\alpha}\right),\quad x,y\in B(0,2),\ t>1.
\end{align*}

Having this estimate  at hand and following the arguments as those in {\bf Case 1}, we can prove the
desired conclusion \eqref{l4-4-1} for the case that $\alpha\in (0,2)$
without the restriction on lower bound of $K_1$.
\end{proof}

\begin{lemma}\label{l4-5}
There is a constant $K_0$ so that if \eqref{e1-1a-} holds for some $K_1\ge K_0$, then for
any $C_0>0$ there exists
positive constants $C_{4}$, $C_{5}$, $C_6$ and $C_7$ such that for all $t>4$ and $x,y\in \R^d$ with $|x-y|>C_0t^{1/2}$,
\begin{equation}\label{l4-5-1}
 p(t,x,y)\ge C_{4}t^{- {d}/{2}}\exp\left(-\frac{C_5|x-y|^2}{t}\right)\exp\left(\max\left\{\frac{C_6 t}{(1+|x|)^\alpha}, C_6t-\frac{C_7(1+|x|^2)}{t}\right\}\right).
\end{equation}
Moreover, if \eqref{e1-1a-} is satisfied with $\alpha\in (0,2)$ and $K_1,K_2\in (0,+\infty)$ $($without any restriction on the lower bound of $K_1)$, then
\eqref{l4-5-1} still holds for $p(t,x,y)$.
\end{lemma}
\begin{proof}
As before we first consider the case that $\alpha\in (0,+\infty)$. Let $U:=B(x,C_0t^{1/2})$. By \eqref{l4-4-1}, for all $t>4$ and $x\in \R^d$,
\begin{align*}
\inf_{z\in U}p\left(\frac{t}{2},x,z\right)\ge c_1t^{-d/2}\exp\left(\max\left\{\frac{c_2 t}{(1+|x|)^\alpha},c_2t-\frac{c_3(1+|x|^2)}{t}\right\}\right).
\end{align*}
On the other hand, for all $t>0$ and $y\in \R^d$,
\begin{align*}
\inf_{z\in U}p\left(\frac{t}{2},z,y\right)&\ge
\inf_{z\in U}q\left(\frac{t}{2},z,y\right)\ge c_4t^{-d/2}\inf_{z\in U}\exp\left(-\frac{c_5|y-z|^2}{t}\right)\ge c_6t^{-d/2} \exp\left(-\frac{c_7|x-y|^2}{t}\right),
\end{align*}
where the last inequality follows from the fact that $\sup_{z\in U}|y-z|\le c_8|x-y|$ due to $|x-y|>C_0t^{1/2}$.
Hence, for all $t>4$ and $x,y\in \R^d$ with $|x-y|>C_0t^{1/2}$,
\begin{align*}
p(t,x,y)&\ge \int_{U}p\left(\frac{t}{2},x,z\right)p\left(\frac{t}{2},z,y\right)\,dz\ge \inf_{z\in U}p\left(\frac{t}{2},x,z\right)\cdot \inf_{z\in U}p\left(\frac{t}{2},z,y\right)\cdot |U|\\
&\ge c_{8}t^{-d/2}\exp\left(-\frac{c_9|x-y|^2}{t}\right)\exp\left(\max\left\{\frac{c_{10} t}{(1+|x|)^\alpha}, c_{10}t-\frac{c_{11}(1+|x|^2)}{t}\right\}\right).
\end{align*}

Now we consider the case that $\alpha\in (0,2)$. By Lemma \ref{l4-4}, we know that in this case \eqref{l4-4-1} holds without any restriction on
the constant $K_1$. Using this fact and following the same arguments as above, we can show \eqref{l4-5-1} when $\alpha\in (0,2)$.
\end{proof}

According to Lemmas \ref{l4-4} and \ref{l4-5}, we have  \begin{proposition}\label{p4-2}
There exists a positive constant $K_0$ so that if \eqref{e1-1a-} holds for some $K_1\ge K_0$, then there are constants $C_8,C_9, C_{10}$ and $C_{11}>0$ so that for all $t>4$ and $x,y\in \R^d$
\begin{equation}\label{p4-2-1}
 p(t,x,y)\ge C_{8}t^{- {d}/{2}}\exp\left(-\frac{C_9|x-y|^2}{t}\right)\exp\left(\max\left\{\frac{C_{10}t}{(1+|x|)^\alpha},C_{10}t-\frac{C_{11}(1+|x|^2)}{t}\right\}\right).
\end{equation} Moreover, if \eqref{e1-1a-} is satisfied with $\alpha\in (0,2)$ and $K_1,K_2\in (0,+\infty)$ $($without any restriction on the lower bound of $K_1)$, then
\eqref{p4-2-1} still holds for $p(t,x,y)$
\end{proposition}

Now, we can present the

\begin{proof}[Proof of Theorem {\rm\ref{t1-2}}] Since $V(x)$ is bounded,  \eqref{note-} holds true. Hence, in order to prove Theorem \ref{t1-2}(1) and (2) we only need to consider the case that $t>1$ large enough. Then the desired assertion immediately follows from Proposition \ref{p4-1} and \ref{p4-2}. Theorem \ref{t1-2}(3) has been proved in Proposition \ref{p4-3}. \end{proof}

Finally, inspired by Remark \ref{remark:one} we would like to understand clearly the relation between the Schr\"{o}dinger heat kernel $p(t,x,y)$ and the Gaussian heat kernel $q(t,x,y)$ for the {\bf Case 2} by considering upper bounds of $T_t^V1(x)$.

\begin{lemma}\label{l4-1}
Suppose that \eqref{e1-1a-} holds. Then there exist constants $C_1,C_2>0$ so that for all $t>0$ and $x\in \R^d$,
\begin{equation}\label{l4-1-1}
\begin{split}
T_t^V1(x)\le C_1\exp\left(C_2\max\left\{\frac{t}{(1+|x|)^\alpha},t-\frac{1+|x|^2}{t}\right\}\right).
\end{split}
\end{equation}
\end{lemma}
\begin{proof}
Without loss of generality we can assume that $|x|\ge 4$; otherwise the assertion is trivial.
Define $U:=B(x,{|x|}/{2})$. By \eqref{e1-4}, we set
\begin{align*}
T_t^V 1(x)&=\Ee_x\left[\exp\left(-\int_0^t V(B_s)\,ds\right)\I_{\{\tau_U>t\}}\right]+
\Ee_x\left[\exp\left(-\int_0^t V(B_s)\,ds\right)\I_{\{\tau_U\le t\}}\right]\\
&=:I_1+I_2.
\end{align*}

According to \eqref{e1-1a-}, it holds that
$$
I_1\le \exp\left(\frac{c_1 t}{(1+|x|)^\alpha}\right).
$$
On the other hand, using \eqref{e1-1a-} again, we have
\begin{align*}
I_2&\le \Ee_x\left[\exp\left(-\int_0^t V(B_s)\,ds\right)\I_{\{\tau_U\le t,B_t\in B(x,{|x|}/{3})\}}\right]\\
&\quad +\Ee_x\left[\exp\left(-\int_0^t V(B_s)\,ds\right)\I_{\{\tau_U\le t,B_t\notin B(x,{|x|}/{3})\}}\right]\\
&\le e^{c_2t}\Big(\Pp_x\left(\tau_U\le t,B_t\in B(x,{|x|}/{3})\right)+\Pp_x\left(\tau_U\le t,B_t\notin B(x,{|x|}/{3})\right)\Big)\\
&=:e^{c_2t}\left(I_{21}+I_{22}\right).
\end{align*}
Furthermore, by the arguments for \eqref{l2-2-3} and \eqref{l2-2-4}, we find
\begin{align*}
I_{21}+I_{22}\le c_3\exp\left(-\frac{c_4(1+|x|^2)}{t}\right).
\end{align*}

Putting all the estimates above together yields the desired conclusion \eqref{l4-1-1}.
\end{proof}

\section{Appendix}
In this section, we recall the following three known lemmas.

\begin{lemma}$($\cite[Lemma 2.1]{CW}$)$\label{l2-6}
Suppose that $U$ is a domain in $\R^d$. Then, for every $x\in U$, $y\notin \bar U$ and $t>0$,
\begin{equation}\label{l2-4-2}
p(t,x,y)=\Ee_x\left[\exp\left(-\int_0^{\tau_U}V(B_s)\,ds\right)\I_{\{\tau_U\le t\}}p\left(t-\tau_U,B_{\tau_U},y\right)\right].
\end{equation}
\end{lemma}

Though Lemma \ref{l2-6} was proved in \cite[Lemma 2.1]{CW} when the potential $V(x)\ge 0$ for all $x\in \R^d$, the proof still works for general locally bounded potentials $V$ that are bounded from below.

\ \

For every domain $D \subset \R^d$,  denote by $q_D(t,x,y)$ the Dirichlet heat kernel of the Brownian motion $\{B_t\}_{t\ge 0}$ on $D$,
and let $\tau_D=\{t>0: B_t\notin D\}$ be the first exit time from
$D$ for $\{B_t\}_{t\ge 0}$.

\begin{lemma}$($\cite[Theorem 1.1]{H}$)$
For every $x\in \R^d$ and $r>0$,
\begin{equation}\label{l2-5-1}
\Pp_x(\tau_{B(x,r)}\in dt, B_{\tau_{B(x,r)}}\in dz)=\frac{1}{2}\frac{\partial q_{B(x,r)}(t,x,\cdot)}{\partial n}(z)\,\sigma(dz)\,dt,
\end{equation}
where $\sigma(dz)$ denotes the Lebesgue surface measure on $\partial B(x,r)$ $($in particular, when $d=1$, $\sigma(dz)$ is the Dirac measure
at the boundary$)$,
and  $\frac{\partial q_{B(x,r)}(t,x,\cdot)}{\partial n}(z)$ denotes the exterior normal derivative
of $q_{B(x,r)}(t,x,\cdot)$ at the point $z\in \partial B(x,r)$.
\end{lemma}

\ \

The following lemma is essentially taken from \cite[Theorem 1.1]{Z} and \cite[Theorem 1]{MS}.
\begin{lemma}\label{l2-1}
There exist positive constants $c_i$, $i=1,\cdots, 4$, such that for every $R>0$, $x\in \R^d$, $t>0$ and $y,z\in B(x,R/2)$,
\begin{equation}\label{l2-1-2}
\begin{split}
\frac{c_1}{t^{d/2}}\exp\left(-c_2\left(\frac{|y-z|^2}{t}+\frac{t}{R^2}\right)\right)\le q_{B(x,R)}(t,y,z)\le \frac{c_3}{t^{d/2}}\exp\left(-c_4\left(\frac{|y-z|^2}{t}+\frac{t}{R^2}\right)\right).
\end{split}
\end{equation}
Moreover, there exist positive constants $c_i$, $i=5,\cdots,8$, so that for any $R>0$, $x\in \R^d$, $t>0$ and $y\in \partial B(x,R)$,
\begin{equation}\label{l2-1-1}
\begin{split}
\frac{c_5R}{t^{d/2+1}}\exp\left(-c_6\left(\frac{R^2}{t}+\frac{t}{R^2}\right)\right)&\le   \frac{\partial q_{B(x,R)}(t,x,\cdot)}{\partial n}(y)
 \le \frac{c_7R}{t^{d/2+1}}\exp\left(-c_8\left(\frac{R^2}{t}+\frac{t}{R^2}\right)\right).
\end{split}
\end{equation}
\end{lemma}
\begin{proof}
By the translation invariant property, it suffices to prove \eqref{l2-1-2} for $x=0$.  According to \cite[Theorem 1.1]{Z} and \cite[Theorem 1]{MS},
for every $y,z\in B(0,1)$ and $t>0$,
\begin{align*}
q_{B(0,1)}(t,x,y)\asymp t^{-d/2}\exp\left(-\frac{|y-z|^2}{t}\right)\cdot
\begin{cases}
\min\left\{1,\frac{(1-|y|)(1-|z|)}{t}\right\},\ & 0<t\le 1,\\
e^{-t}(1-|y|)(1-|z|),\ & t>1.
\end{cases}
\end{align*}
This along with the scaling property
\begin{align*}
q_{B(0,1)}(t,y,z)=R^dq_{B(0,R)}(R^2t, Ry,Rz)
\end{align*}
yields the desired estimates \eqref{l2-1-2}.

The estimate \eqref{l2-1-1} has been proved by \cite[Lemma 2.3]{CW}, and so we omit it here. \end{proof}

\ \

\noindent {\bf Acknowledgements.}\,\,  The research of Xin Chen is supported by National Natural Science Foundation of China
(Nos. 12122111).
The research
of Jian Wang is supported by the National Key R\&D Program of China (2022YFA1006003) and the National Natural Science Foundation of China (Nos. 12071076 and 12225104).












\begin{thebibliography}{99}
\bibitem{BK}
Baraniewicz, M. and Kaleta, K.:
Integral kernels of Schr\"{o}dinger semigroups with nonnegative locally bounded potentials, arXiv:2302.13886.


\bibitem{BFT}
Barbatis, G., Filippas, S. and Tertikas, A.:
Critical heat kernel estimates for Schr\"odinger operators via Hardy-Sobolev inequalities,
\emph{J. Funct. Anal.}, {\bf 208} (2004), 1--30.

\bibitem{BY} Benguria, R. and Yarur, C.:
Sharp condition on the decay of the potential for the absence of a zero-energy
ground state of the Schr\"odinger equation, \emph{J. Phys. A}, {\bf 23} (1990), 1513--1518.

\bibitem{BDS} Bogdan, K.,  Dziuba\'{n}ski, J. and  Szczypkowski, K.: Sharp Gaussian estimates for heat kernels of Schr\"{o}dinger operators,
\emph{Integral Equ. Oper. Theory}, {\bf91} (2019), paper no.\ 3.



\bibitem{BJS}
Bogdan, K., Jakubowski, T. and Sydor, S.: Estimates of perturbation series for kernels, \emph{J. Evol.
Equ.}, {\bf12} (2012), 973--984.

\bibitem{BS}
Bogdan, K. and Szczypkowski, K.: Gaussian estimates for Schr\"{o}dinger perturbations, \emph{Studia Math.},
{\bf221} (2014), 151--173.


\bibitem{CW} Chen, X. and Wang, J.: Two-sided heat kernel estimates for Schr\"{o}dinger operators
with unbounded potentials, to appear in \emph{Ann. Probab.}, see also arXiv:2301.06744.


\bibitem{CK} Christ, M. and Kiselev, A.: One-dimensional Schr\"odinger operators with slowly decaying potentials: spectra
and asymptotics, notes for IPAM tutorial, 2001 Workshop on Oscillatory Integrals and Dispersive Equations, 2001.


\bibitem{CZ}  Chung, K.L. and  Zhao, Z.:
\emph{From Brownian Motion to Schr\"odinger's Equation},
Grundlehren der Mathematischen Wissenschaften, vol. {\bf 312}. Springer-Verlag, Berlin, 1995.



\bibitem{DS91}
Davies, E.B. and Simon, B.: $L^p$-norms of non-critical for Schr\"{o}dinger semigroups, \emph{J. Funct. Anal.}, {\bf 102} (1991), 95--115.

\bibitem{DK} Denisov, S. and Kiselev, A.: Spectral properties of Schr\"odinger operators with decaying potentials, in: \emph{Spectral Theory and Mathematical Physics: a Festschrift in Honor of Barry Simon's 60th Birthday},
vol. 76 of Proc. Sympos. Pure Math., Amer. Math. Soc., Providence, RI, 2007, pp. 565--589.


\bibitem{DS} Derezi\'nski, J. and Skibsted, E.: Quantum scattering at low energies, \emph{J. Funct. Anal.}, {\bf 257} (2009), 1828--1920.

\bibitem{EK} Eastham, M.S.P. and Kalf, H.: \emph{Schr\"odinger-type Operators with Continuous Spectra}, Research Notes in Mathematics, vol.\ 65, Pitman Advanced Publishing Program, London, 1982.

\bibitem{FMT} Filippas, S., Moschini, L. and Tertikas, A.:
Sharp two-sided heat kernel estimates for critical Schr\"odinger operators on bounded domains,
\emph{Comm. Math. Phys.}, {\bf 273} (2007), 237--281.

\bibitem{FO} Fournais, S. and Skibsted, E.: Zero energy asymptotics of the resolvent for a class of slowly decaying potentials,
\emph{Math. Z.}, {\bf 248} (2004), 593--633.


\bibitem{H} Hsu, E.P.: Brownian exit distribution of a ball, in: \emph{Seminar on Stochastic Processes}, {Progr. Probab. Statist.}, vol. {\bf12},
Birkh\"auser Boston, 1986, pp. 108--116.

\bibitem{IKO}
Ishige, K., Kabeya, Y. and Ouhabaz, E.M.: The heat kernel of a Schr\"{o}dinger operator with inverse sequare
potential, \emph{Proc. London Math. Soc.}, {\bf115} (2017), 381--410.

\bibitem{JK} Jensen, A. and Kato, T.: Spectral properties of Schr\"odinger operators and time-decay of the wave functions,
\emph{Duke Math. J.}, {\bf 46} (1979), 583--611.

\bibitem{KL} Kaleta, K. and L\"orinczi, J.: Zero-energy bound state decay for non-local
Schr\"odinger operators, \emph{Comm. Math. Phys.}, {\bf 374} (2020), 2151--2191.

\bibitem{KT} Koch, H. and Tataru, D.: Sharp counterexamples in unique continuation for second order elliptic equations,
\emph{J. Reine Angew. Math.}, {\bf 542} (2002), 133--146.

\bibitem{Lis}
Liskevich, V. and Semenov, Y.A.: Two-sided estimates of the heat kernel of the Schr\"{o}dinger operator,  \emph{Bull. Lond. Math.
Soc.}, {\bf 30} (1998), 596--602.

\bibitem{MS}  Malecki, J. and Serafin, G.: Dirichlet heat kernel for the Laplacian in a ball,
\emph{Potential Anal.}, {\bf 52} (2020), 545--563.

\bibitem{MS1}
Milman, P.D. and Semenov, Y.A.: Heat kernel bounds and desingularizing weights, \emph{J. Funct.
Anal.}, {\bf202} (2003), 1--24.


\bibitem{MS2}
Milman, P.D. and Semenov, Y.A.: Global heat kernel bounds via desingularizing weights, \emph{J. Funct.
Anal.}, {\bf212} (2004), 373--398. Corrigendum: \emph{J. Funct. Anal.}, {\bf229} (2005), 238--239.


\bibitem{Mu} Murata, M: Asymptotic expansions in time for solutions of Schr\"odinger-type equations,
\emph{J. Funct. Anal.}, {\bf 49} (1982), 10--56.

\bibitem{N}
Nakamura, S.: Low-energy asymptotics for Schr\"odinger operators with slowly decreasing potentials,
\emph{Comm. Math. Phys.}, {\bf 161} (1994), 63--76.


\bibitem{RS} Reed, M. and Simon, B.: \emph{Methods of Modern Mathematical Physics}, vol.\, 3-4, Academic Press, London,
1979.

\bibitem{S}
Semenov, Y.A.: Stability of $L_p$-spectrum of generalized Schr\"{o}dinger operators and equivalence of
Green's functions, \emph{International Mathematics Research Notices}, {\bf 12} (1997), 573--593.

\bibitem{Simon1}
Simon, B.: Schr\"{o}dinger semigroups, \emph{Bull. Am. Math. Soc. $($N.S.$)$}, {\bf7} (1982), 447--526.


\bibitem{SW} Skibsted, E. and Wang, X.P.: Two-body threshold spectral analysis, the critical case,
\emph{J. Funct. Anal.}, {\bf 260} (2011), 1766--1794.


\bibitem{Y} Yafaev, D.: The low energy scattering for slowly decreasing potentials, \emph{Comm. Math. Phys.}, {\bf 85}
(1982), 177--196.

\bibitem{Z3} Zhang, Q.-S.:  Gaussian bounds for the fundamental solutions of  $\nabla(A\nabla u)+B\nabla u-u_t=0$,
\emph{Manuscripta Math.}, {\bf 93} (1997), 381--390.

\bibitem{Z1} Zhang, Q.-S.:
Large time behavior of Schr\"odinger heat kernels and applications,
\emph{Comm. Math. Phys.}, {\bf 210} (2000), 371--398.


\bibitem{Z2} Zhang, Q.-S.:
Global bounds of Schr\"odinger heat kernels with negative potentials,
\emph{J. Funct. Anal.}, {\bf 182} (2001), 344--370.

\bibitem{Z} Zhang, Q.-S.:
The boundary behavior of heat kernels of Dirichlet Laplacians, \emph{J. Differential Equations}, {\bf 182} (2002), 416--430.



\bibitem{Z3} Zhang, Q.-S.:
A sharp comparison result concerning Schr\"{o}dinger heat kernels,  \emph{Bull. Lond. Math. Soc.},  {\bf35} (2003),
461--472.

\end{thebibliography}
\end{document}